\documentclass[11pt]{amsart}
\usepackage{amscd}
\usepackage[arrow,matrix]{xy}
\usepackage{graphicx}
\usepackage{amsmath}
\usepackage{hyperref}
\usepackage{amsrefs}
\usepackage{xcolor}
\usepackage{soul}
\usepackage{comment}
\usepackage{amsmath, latexsym, amssymb}
\input xypic

\numberwithin{equation}{section}
\theoremstyle{plain}
\newtheorem{lemma}{Lemma}[section]
\newtheorem{proposition}[lemma]{Proposition}
\newtheorem{theorem}[lemma]{Theorem}
\newtheorem{corollary}[lemma]{Corollary}

\theoremstyle{definition}
\newtheorem{definition}[lemma]{Definition}
\newtheorem{remark}[lemma]{Remark}
\newtheorem{example}[lemma]{Example}

\usepackage{tikz}
\usepackage{dsfont}
\usetikzlibrary{matrix}

\definecolor{grey}{RGB}{188,188,188}

\begin{document}
\newcommand{\R}{{\mathbb R}}
\newcommand{\C}{{\mathbb C}}
\newcommand{\F}{{\mathbb F}}
\renewcommand{\O}{{\mathbb O}}
\newcommand{\Z}{{\mathbb Z}} 
\newcommand{\N}{{\mathbb N}}
\newcommand{\Q}{{\mathbb Q}}
\renewcommand{\H}{{\mathbb H}}
\newcommand{\X}{{\mathfrak X}}

\newcommand{\Aa}{{\mathcal A}}
\newcommand{\Bb}{{\mathcal B}}
\newcommand{\Cc}{{\mathcal C}}    
\newcommand{\Dd}{{\mathcal D}}
\newcommand{\Ee}{{\mathcal E}}
\newcommand{\Ff}{{\mathcal F}}
\newcommand{\Gg}{{\mathcal G}}    
\newcommand{\Hh}{{\mathcal H}}
\newcommand{\Kk}{{\mathcal K}}
\newcommand{\Ii}{{\mathcal I}}
\newcommand{\Jj}{{\mathcal J}}
\newcommand{\Ll}{{\mathcal L}}    
\newcommand{\Mm}{{\mathcal M}}    
\newcommand{\Nn}{{\mathcal N}}
\newcommand{\Oo}{{\mathcal O}}
\newcommand{\Pp}{{\mathcal P}}
\newcommand{\Qq}{{\mathcal Q}}
\newcommand{\Rr}{{\mathcal R}}
\newcommand{\Ss}{{\mathcal S}}
\newcommand{\Tt}{{\mathcal T}}
\newcommand{\Uu}{{\mathcal U}}
\newcommand{\Vv}{{\mathcal V}}
\newcommand{\Ww}{{\mathcal W}}
\newcommand{\Xx}{{\mathcal X}}
\newcommand{\Yy}{{\mathcal Y}}
\newcommand{\Zz}{{\mathcal Z}}

\renewcommand{\a}{{\mathfrak a}}
\renewcommand{\b}{{\mathfrak b}}
\newcommand{\e}{{\mathfrak e}}
\renewcommand{\k}{{\mathfrak k}}
\newcommand{\m}{{\mathfrak m}}
\newcommand{\pg}{{\mathfrak p}}
\newcommand{\g}{{\mathfrak g}}
\newcommand{\gl}{{\mathfrak gl}}
\newcommand{\h}{{\mathfrak h}}
\renewcommand{\l}{{\mathfrak l}}
\newcommand{\sm}{{\mathfrak m}}
\newcommand{\n}{{\mathfrak n}}
\newcommand{\s}{{\mathfrak s}}
\renewcommand{\o}{{\mathfrak o}}
\renewcommand{\u}{{\mathfrak u}}
\newcommand{\su}{{\mathfrak su}}

\newcommand{\ssl}{{\mathfrak sl}}
\newcommand{\ssp}{{\mathfrak sp}}
\renewcommand{\t}{{\mathfrak t }}

\newcommand{\zt}{{\tilde z}}
\newcommand{\xt}{{\tilde x}}
\newcommand{\Ht}{\widetilde{H}}
\newcommand{\ut}{{\tilde u}}
\newcommand{\Mt}{{\widetilde M}}
\newcommand{\Llt}{{\widetilde{\mathcal L}}}
\newcommand{\yt}{{\tilde y}}
\newcommand{\vt}{{\tilde v}}
\newcommand{\Ppt}{{\widetilde{\mathcal P}}}
\newcommand{\bp }{{\bar \partial}} 

\newcommand{\Remark}{{\it Remark}}
\newcommand{\Proof}{{\it Proof}}
\newcommand{\ad}{{\rm ad}}
\newcommand{\Om}{{\Omega}}
\newcommand{\om}{{\omega}}
\newcommand{\eps}{{\varepsilon}}
\newcommand{\Di}{{\rm Diff}}

\renewcommand{\a}{{\mathfrak a}}
\renewcommand{\b}{{\mathfrak b}}
\renewcommand{\k}{{\mathfrak k}}
\renewcommand{\l}{{\mathfrak l}}
\renewcommand{\o}{{\mathfrak o}}
\renewcommand{\u}{{\mathfrak u}}
\renewcommand{\t}{{\mathfrak t }}
\newcommand{\Cinf}{C^{\infty}}
\newcommand{\la}{\langle}
\newcommand{\ra}{\rangle}
\newcommand{\half}{\scriptstyle\frac{1}{2}}
\newcommand{\p}{{\partial}}
\newcommand{\notsub}{\not\subset}
\newcommand{\iI}{{I}}               
\newcommand{\bI}{{\partial I}}      
\newcommand{\LRA}{\Longrightarrow}
\newcommand{\LLA}{\Longleftarrow}
\newcommand{\lra}{\longrightarrow}
\newcommand{\LLR}{\Longleftrightarrow}
\newcommand{\lla}{\longleftarrow}
\newcommand{\INTO}{\hookrightarrow}

\newcommand{\QED}{\hfill$\Box$\medskip}
\newcommand{\UuU}{\Upsilon _{\delta}(H_0) \times \Uu _{\delta} (J_0)}
\newcommand{\bm}{\boldmath}
\newcommand{\pb}{{\mathbf p}}

\newcommand{\GL}{{\rm GL}}
\newcommand{\SL}{{\rm SL}}
\newcommand{\SO}{{\rm SO}}
\newcommand{\G}{{\rm G_2}}
\newcommand{\Spin}{{\rm Spin(7)}}
\newcommand{\Id}{\mathrm{Id}}
\newcommand{\smallr}{\mathrm{small}}

\newenvironment{nouppercase}{%
  \let\uppercase\relax%
  \renewcommand{\uppercasenonmath}[1]{}}{}

\setcounter{tocdepth}{1}

\title[$C_\infty$-algebras and $\ell(r-1)+2$-dimensional $(r-1)$-connected manifolds]{\Large {Unital $C_\infty$-algebras and the real homotopy type of $(r-1)$-connected compact manifolds of dimension $\le \ell(r-1)+ 2$}}

	\author{Domenico Fiorenza}
\address{
	Dipartimento di Matematica ``Guido Castelnuovo'',
	Universit\`a   di Roma ``La Sapienza",
	Piazzale Aldo Moro 2, 00185 Roma, 
	Italy}
\email{fiorenza@mat.uniroma1.it} 

\author{H\^ong V\^an L\^e }
\address{Institute  of Mathematics of the Czech Academy of Sciences,
	Zitna 25, 11567  Praha 1, Czech Republic}

\address{{\it and} Charles University, Faculty of Mathematics and Physics\\
	Ke Karlovu 3, 121 16 Praha 2, Czech Republic}
	
\email{hvle@math.cas.cz}
\date{\today}

\date{\today}

\abstract We encode  the  real homotopy  type of an $n$-dimensional  $(r-1)$-connected compact manifold $M$, $ r\ge  2$ into
  a  minimal unital $C_\infty$-structure  on $H^* (M,\R)$,  obtained via  a Hodge homotopy transfer of the unital   DGCA structure  of  the small  quotient  algebra associated  with a  Hodge decomposition  of the  de Rham algebra $\Aa^*(M)$, which has been proposed by  Fiorenza-Kawai-L\^e-Schwachh\"ofer   in \cite{FKLS2021}. 
We prove that if  $n \le$ $\ell(r-1)+2$, with $\ell\geq 4$, the multiplication $\mu_k$ on  the   minimal  unital $C_\infty$-algebra $H^*(M,\R)$  vanishes   for all $k \ge \ell-1$. This extends the results from  \cite{FKLS2021}, extending the bound on the dimension from $5r-3$ the general bound $\ell(r-1)+2$.   We  also  prove a variant  of this  result, conjectured by Zhou, stating that if $n \le  \ell(r-1)+4$  and   $b_r (M) =1$  then   the multiplication $\mu_k$   for all $k \ge \ell-1$  vanishes.  This  implies two formality results by Cavalcanti \cite{Cavalcanti2006}. We show that in any dimension $n$ the Harrison cohomology class $[\mu_3]\in {\rm HHarr}^{3,-1}(H^* (M, \R), H^*(M, \R)) $ is a homotopy invariant of $M$ and the first obstruction to formality,  and provide a detailed proof that if $n\leq 4r-1$ this is  the only obstruction.
Furthermore, we show that in any dimension $n$  the class  $[\mu_3]$  and the   Bianchi-Massey  tensor  invented by Crowley-Nordstr\"om in \cite{CN} define  each other uniquely. 
\endabstract

\begin{nouppercase}
\maketitle
\end{nouppercase}

\tableofcontents

\section{Introduction}

In  his seminal  papers  \cite{Sullivan1975, Sullivan1977}, Sullivan  provided algebraic models  of  simply connected  compact manifolds $M$  of dimension  at least 5,  using   the smooth differential forms  on $M$.\footnote{The dimension 5  is motivated by  Smale's solution of  the  Poincar\'e  conjecture   in dimension  at least 5   \cite{Smale1961}, which opened  the way  to classification  of simply connected differentiable manifolds in dimension 5 \cite{Smale1962}.}  Sullivan's theorem \cite[Theorem 13.1]{Sullivan1977} states that  the diffeomorphism  type of  a  simply connected  manifold $M$ of dimension  at  least  $5$  is  defined  by  its real  homotopy  type, real Pontryagin classes, and by the torsion of $M$, up to   finite  ambiguity   \cite[Theorem 10.4]{Sullivan1977}, see  also   the extension  of Sullivan's theorem by Kreck and Triantafillou \cite[Theorem 1.1]{KT1991}.  Furthermore,    Sullivan proved that rational homotopy type of $M$ is captured by the DGCA weak equivalence class of piecewise polynomial differential forms  with rational coefficients on $M$ and  the real homotopy type by that of de Rham forms  \cite[Theorems A, B, C, D]{Sullivan1975}. 

  In 1979  Miller  showed  that  any  $(r-1)$-connected  compact manifold of dimension  less than or equal to $4r-2$ is formal \cite{Miller1979}.  In particular,   the real  homotopy type  of any  compact  simply connected  manifold  $M$  of dimension less than  or equal 6 is defined  by  its cohomology  ring $H^* (M, \R)$. In 2020  Crowley-Nordstr\"om, using    Sullivan minimal models,   showed  that the rational homotopy type   of   $(r-1)$-connected
compact  manifolds $M$ of dimension less  than  or equal $(5r-3)$, $ r\ge 2,$  is determined  by its  cohomology algebra $H^* (M, \Q)$ and by the Bianchi-Massey  tensor  they introduced in \cite{CN}. In  the  recent  paper \cite{NN2021} Nagy and Nordstr\"om   developed  the  techniques from \cite{CN}  further and showed  that   the formality  of  a $(r-1)$-connected compact manifold  $M$ of dimension  up to  $(5r-2)$ is  equivalent to the vanishing  of  the  Bianchi-Massey tensor and  the pentagon  Massey tensor,  a new homotopy invariant  introduced  by  them.

In 2021  Fiorenza-Kawai-L\^e-Schwachh\"ofer   proposed a new method  to  study the real  homotopy   type  of a  compact simply connected  smooth manifold by introducing  the   concept  of a Poincar\'e  DGCA   admitting a Hodge  homotopy  \cite{FKLS2021}, see  also  Definitions \ref{def:pdgca}, \ref{def:hodgehomotopy} below.  They made the crucial observation that a simply connected Poincar\'e DGCA  $\Aa^*$ of degree $n$ admitting a Hodge homotopy  is weak  equivalent  to a finite  dimensional  non-degenerate  Poincar\'e  DGCA $\Qq^* (\Aa^*)$ of degree $n$ endowed  with a Hodge homotopy \cite[Theorem 3.10]{KS2001}, recalled in Proposition \ref{prop:qsmall} below.  Using  a Hodge  homotopy transfer  of the  DCGA  structure  of $\Qq^* (\Aa^*)$  to  the  cohomology  algebra $H^* (\Qq^*  (\Aa^*)) = H^* (\Aa^*)$, they  proved  that if  $\Aa^*$ is a $n$-dimensional $(r-1)$-connected  Poicar\'e  DGCA admitting a  Hodge  homotopy,  $n \le 5r-3$, $ r\ge 2$, then $\Aa^*$ is $C_\infty$-quasi-isomorphic to a minimal $C_\infty$-algebra with  vanishing   higher multiplications $m_k$,  for  $k \ge 4$ \cite[Theorem 4.1]{FKLS2021}. By Kadeishivili results
 \cite{Kadeishvili1988, Kadeishvili1993, Kadeishvili2009}, see also \cite[Theorems 14, 15]{Kadeishvili2023},   the weak  equivalence  class   of the   DGCA  of  piecewise polynomial differential forms on $M$  with rational coefficients on $M$ and that of the de Rham  forms  is captured  by the  homotopy  type of the  associated  $C_\infty$-algebra obtained  by   a  homotopy transfer. Thus  Fiorenza-Kawai-L\^e-Schwachh\"ofer's results     provide  new information of   real homotopy type  of  $n$-dimensional $(r-1)$-connected    compact manifold $M$, $ n \le 5r-3$, $r \ge 2$,  in terms  of   a minimal $C_\infty$-algebra  structure on $H^*(M, \R) $. Next, Zhou improved the result from \cite{FKLS2021} showing that generally, for any $\ell\geq 4$, if $n\leq \ell (r-1)+2$, then  $\Aa^*$ is $A_\infty$-quasi-isomorphic to a minimal $A_\infty$-algebra with  vanishing   higher multiplications $m_k$,  for  $k \ge \ell-1$.

In this  paper we  show the effectiveness of the techniques from \cite{FKLS2021} to give a very quick proof of Zhou's result, as well as extending it from $A_\infty$ to $C_\infty$-algebras (Theorem \ref{thm:A2-algebra}), and   develop   the techniques   in \cite{FKLS2021} further 
to  study  the real  homotopy  type  of  an  $n$-dimensional $(r-1)$-connected    compact manifold $M$,  $r \ge 2$,  in terms  of   a minimal unital $C_\infty$-algebra  structure on $H^*(M, \R)$,    whose multiplication $\mu_2$ coincides  with the  usual multiplication $\cdot$ on $H^* (M,\R)$.  As a result, we  proved  a conjecture by Zhou  stating   that if $n \le  \ell (r-1)+4$  and  $b_r (M) = 1$  then  $m_k$ vanishes for $k \ge \ell -1$ \cite{Zhou2019}.  This  is a  generalization of  two  formality  results  by Cavalcanti \cite{Cavalcanti2006}. We  show that the  Harrison cohomology  class $[\mu_3] \in {\rm HHarr}^{3,-1} (H^*(M, \R), H^* (M, \R))$, where $\mu_3$ is the ternary  operation on the  unital  $C_\infty$-algebra $(H^*(M,\R),  \cdot, \mu_3, \ldots)$,  is a  homotopy  invariant  of   $M$ and the 
first obstruction to its  formality. Moreover we show  that $[\mu_3]$  and the  Bianchi-Massey tensor  define  each other uniquely. For $\ell=6$ one finds  that  if $n \le 6r-4$  then  the  multiplication $\mu_k$ on the  unital  $C_\infty$-algebra $(H^*(M,\R),  \cdot, \mu_3, \ldots)$ vanishes for $k>4$, so that the unital  $C_\infty$-algebra structure is completely encoded into the multiplications $\mu_3$ and $\mu_4$. This  complements to Nagy-Nordstr\"om's  result  on   the equivalence  of the   formality  of  $(r-1)$-connected  compact  manifolds  of dimension  $n \le  5r-2$ and the  vanishing  of the  Bianchi-Massey and   pentagon Massey tensor  \cite[Theorem 1.7]{NN2021}. 

Our paper  is organized  as follows.  In  Section \ref{sec:ainftyreg},  we  recall  the concept  of an $(r-1)$-connected Poincar\'e DGCA $\Aa^*$ of degree $n$ admitting a Hodge homotopy (Definitions \ref{def:pdgca}, \ref{def:hodgehomotopy} {and \ref{def:connected}})  and  the related  concept of  the small quotient  algebra
$\Qq^*(\Aa^*)$  (Proposition \ref{prop:qsmall}).   Using  a Hodge homotopy transfer of the  unital DGCA  structure on $\Qq^*(\Aa^*)$ to a minimal unital  $C_\infty$-structure  on a   space  $\Hh^*$ of   harmonic elements of  $\Aa^*$, we   show   that  if   $n \le \ell (r-1)+2$ with $\ell\geq 4$, then  all   the higher  operations $m_k$, $k>\ell$, on the  minimal unital  $C_\infty$-algebra $(\Hh^*, m_2, m_3, \ldots, )$  vanish; in particular, $m_4$ vanishes  if $n \le 5r-3$ and  $m_3$ vanishes  if $n \le 4r-2$ (Theorem \ref{thm:A2-algebra}). We  then prove the Zhou's conjecture  (Theorem \ref{thm:zhou-s-conjecture}) and derive  from it and from Theorem \ref{thm:A2-algebra} various   previously known    formality  results.  In Section \ref{sec:harrform},  we transfer the operations $m_k$ on $\mathcal{H}^\ast$ to operations $\mu_k$ on $H^\ast(\Aa^\ast)$ via the canonical linear isomorphism $\mathcal{H}^\ast\xrightarrow{\sim} H^\ast(\Aa^\ast)$ and  prove that  the Harrison cohomology class $[\mu_3]\in  {\rm HHarr}^{3,-1} (H^*(\Aa^*), H^* (\Aa^*))$ is  a homotopy invariant  of  $\Aa^*$ (Lemma \ref{lemma:canonical})  and the first obstruction to its formality (Lemma \ref{lemma:first-obs-formality}). Furthermore, $H^* (\Aa^*)$ and $[\mu_3]$ are  complete  invariants  of  the DGCA homotopy class of  an $(r-1)$ connected  Poincar\'e DGCA of degree $n \le 4r-1$ admitting a Hodge  homotopy (Theorem \ref{thm:complete-invariant}).  In Section \ref{sec:HarrBM}, using the spectral sequence symmetric-to-Harrison, we  establish  the equivalence  between  the   Harrison cohomology class $[\mu_3]$ and the Bianchi-Massey  tensor (Theorem \ref{thm:HarrBM}).  In the last Section \ref{sec:conclusions}, we summarize  our results  and  suggest  some  problems  for  future   investigations.

\section{The small quotient algebra  and the unital $C_\infty$-algebra   of a  $(r-1)$ homologically connected ($r>1$)  compact manifold of dimension  up to $(\ell(r-1)+2)$ }\label{sec:ainftyreg}
 In this section, we briefly   recall the  concept  of  a Poincar\'e DCGA  admitting a Hodge  homotopy, introduced  and investigated by Fiorenza-Kawai-L\^e-Schwachh\"ofer in \cite{FKLS2021}, which  shall be utilized   to study unital $C_\infty$-algebras associated with $(r-1)$-connected  compact   manifolds of dimension up to $\ell(r-1)+2$, with $\ell\geq 4$,  in our paper. Then under the assumption the algebra admits a Hodge homotopy\footnote{This assumption is always verified by the de Rham algebra of a compact oriented manifold.} we give a very quick proof of    a  theorem  by Zhou on the homotopy type  of  $(r-1)$ connected  Poincar\'e DGCAs  of degree  less than or equal to  $\ell(r-1)+2$    (Theorem \ref{thm:A2-algebra}).   This extend the analogous result in \cite{FKLS2021},  the  real (or  rational) homotopy  type   of a   smooth  that was derived under the assumption of degree  less than or equal to  $5r-3$, corresponding to $\ell=5$. We  also  give a proof  to a   Zhou's  conjecture (Theorem \ref{thm:zhou-s-conjecture})  and derive  from it  Cavalcanti's formality results (Corollary \ref{prop:cavalcanti1}).

\subsection{Poincar\'e-DGCA  of  Hodge type}\label{subs:PDGCA-Hodge}
Given  a  vector  space $V$ over   a field  $\mathbb{K}$ we denote by $V^\vee$ the dual space of $V$.
\begin{definition}\label{def:pdgca} cf. \cite[Def. 4.1]{LS04}, cf. \cite[Def. 2.7]{CN}, cf. \cite[Def. 2.1]{FKLS2021}. A  {\it Poincar\'e   DGCA $(\Aa^*, d)$ of degree  $n$}  over a  field  $\mathbb{K}$  is  a  DGCA
	$\Aa^* = \oplus_{k =0}^n \Aa^k$ whose  cohomology ring $H^* (\Aa)$ is finite  dimensional over $\mathbb{K}$ and has     a {\it fundamental class}
	$\int  \in  (H^n (\Aa))^\vee$  such that  the pairing, defined  by using the cup-product,
\begin{align}\label{eq:pairingh}
\la \alpha  ^k, \beta ^l \ra = \begin{cases} \int \alpha ^k \cdot \beta ^l    & \text{ if } k +l = n,\\
0  & \text{ else, } 
\end{cases}
\end{align}	
is non-degenerate, i.e.  $\la \alpha, H^* (\Aa) \ra =0$ if and only if  $\alpha = 0$. 
\end{definition}

The  pairing    on $H^* (\Aa)$  induces   a  pairing  on $\Aa^*$ as follows
\begin{align}\label{eq:pairing}
\la \alpha  ^k, \beta ^l \ra = \begin{cases} \int [\alpha ^k \cdot \beta ^l ]   & \text{ if } k +l = n,\\
 0  & \text{ else, } 
\end{cases}
\end{align}
where   $[\cdot ]$ stands for  the projection $ \Aa^n\supseteq \ker d \to H^n (\Aa^*)$. {The  Poincar\'e  DCGA  $(\Aa^*,d)$  is called {\it  non-degenerate}, if the   pairing $\la \cdot, \cdot \ra$ on $\Aa^*$ is non-degenerate.}

{
\begin{definition}\label{def:hodgehomotopy}
	Let $(\Aa^\ast, d, \cdot, \la -,- \ra)$ be a Poincar\'e DGCA. A \emph{Hodge homotopy} for $(\Aa^\ast, d, \cdot, \la -,- \ra)$ is a degree $-1$ operator $d^-\colon \Aa^\ast\to \Aa^{\ast-1}$ such that:
	\begin{equation}\label{eq:commute}
	d^-d^-=0;\quad  d^-dd^-=d^-;\quad d d^- d =d;
	\end{equation}
	\begin{equation}\label{eq:orthogonality}
	\langle \mathrm{Im}(d^-), \mathrm{Im}(d^-)\rangle =0;
	\quad 
	\langle \mathrm{Im}(\pi_{\mathcal{H}^\ast}), \mathrm{Im}(d^-)\rangle =0,
	\end{equation}
	where 
	\begin{equation}\label{eq:proj} \pi_{\mathcal{H}^\ast}:=\mathrm{id}_{\Aa^\ast}-[d,d^{-1}]. 	\end{equation}
\end{definition}	

\begin{remark}
It follows from \eqref{eq:commute} that the two endomorphisms $dd^-,d^-d\colon \Aa^\ast\to \Aa^\ast$ are orthogonal idempotents. So they both, and also  $\pi_\Hh$, are projection operators.
\end{remark}
\begin{definition}	
	The graded subspace $\mathcal{H}^\ast:=\mathrm{Im}(\pi_{\mathcal{H}^\ast})$ is called the \emph{harmonic subspace} for the Hodge homotopy $d^{-}$. \end{definition}
}


\begin{remark}
It follows from the definition that $d\pi_{\mathcal{H}^\ast}=\pi_{\mathcal{H}^\ast}d=d^-\pi_{\mathcal{H}^\ast}=\pi_{\mathcal{H}^\ast}d^-=0$, and that we have a direct sum decomposition
\[
\Aa^\ast= \mathcal{H}^\ast\oplus\underbrace{dd^-\Aa^\ast\oplus d^-d\Aa^\ast}_{{\mathcal{L}_{\mathcal{A}}^\ast}}
\]
with $({\mathcal{L}_{\mathcal{A}}^\ast},d)$ an acyclic complex and $d^-({\mathcal{L}_{\mathcal{A}}^\ast})\subseteq {\mathcal{L}_{\mathcal{A}}^\ast}$. Also, the composition
\[
\mathcal{H}^\ast\to \ker(d) \to H^\ast(\Aa)
\]
is an isomorphism of graded vector spaces.
\end{remark}

\begin{remark}
Since $\int d\omega=0$, the pairing on $\mathcal{A}^\ast$ satisfies
\[
\langle d\alpha, \beta\rangle =-(-1)^{\mathrm{deg}(\alpha)}\langle \alpha, d\beta\rangle.
\]
This implies that the direct sum decomposition $\Aa^\ast= \mathcal{H}^\ast\oplus{\mathcal{L}_{\mathcal{A}}^\ast}$ is a orthogonal direct sum decomposition: $\Aa^\ast= \mathcal{H}^\ast\oplus^\perp{\mathcal{L}_{\mathcal{A}}^\ast}$.
\end{remark}

\begin{example} \cite[Remark 2.7 (1)]{FKLS2021}\label{ex:d-}
	Let  $(M,g)$ be an $n$-dimensional  compact oriented Riemannian manifold, let $\Aa^* = \Aa^* (M)$ the de Rham algebra of $M$, let $\Hh^*$  be  the   space of harmonic forms  and let  $G = \Delta_g ^{-1}$ be the Green operator of $(M,g)$. Then integration over $M$ makes $\Aa^*$ a Poincar\'e DGCA of degree $n$, and 
the operator $d^-$ defined by
\[
\begin{cases}
d^{-}\bigr\vert_{\Hh^*} = 0 \\ \\ 
d^{-}\bigr\vert_{d \Aa^*  \oplus d^* \Aa} = G d^*
\end{cases}
\]	
	is a Hodge homotopy. 
\end{example}
\begin{remark}
A richer structure is obtained if instead of a single Hodge homotopy one has a pair of compatible Hodge homotopies. This is what happens, for instance, for K\"ahler manifolds, and the essential ingredient in  Deligne-Griffiths-Morgan-Sullivan proof of formality of compact K\"ahler manifolds \cite{DGMS1975}.
\end{remark}
\begin{definition}\label{def:connected}
A DGCA $\Aa^\ast$ is said to be $(r-1)$-connected, with $r\geq  2$ if $H^0(\Aa)=\mathbb{K}$ and $H^k(\Aa)=0$ for $1\leq k\leq r-1$.
\end{definition}
A crucial result from \cite{FKLS2021} is the following.
\begin{proposition}\label{prop:qsmall} (cf. \cite[Theorem  3.10,  (4.1)]{FKLS2021}).  Let $r \ge 2$ and let 
	  $\Aa^*$ be a $(r-1)$-connected   Poincar\'e  DGCA  of degree  $n$. If $\Aa^*$ admits a Hodge homotopy, then  there exists a nondegenerate finite dimensional Poincar\'e  DGCA  $(\mathcal{Q}^\ast,{d_{\Qq}})$ of degree  $n$ endowed with a Hodge homotopy ${d_{\Qq}^-}$, homotopy equivalent to $\Aa^\ast$ via an explicit zig-zag of quasi-isomorphisms and such that
\begin{equation}\label{eq:qsmall}
\begin{cases}
\mathcal{Q}^0=\mathbb{K}\\
\mathcal{Q}^k=0, & 1\leq k\leq r-1\\
\mathcal{Q}^k=\mathcal{H}^k, & r\leq k\leq 2r-2\\
\mathcal{Q}^{k}=\mathcal{H}^{k}\oplus {\Ll_\Qq^k} &  2r -1 \le k \le n-2r +1\\
\mathcal{Q}^k=\mathcal{H}^k, & n -2r +2\leq k\leq n-r\\
\mathcal{Q}^k=0, &   n-r+1\leq k \le  n-1\\
\mathcal{Q}^{n}=\mathcal{H}^n\cong \mathbb{K},
\end{cases}
\end{equation}
{with $\Ll^\ast_\Qq=d_\Qq d_\Qq^-\Qq^\ast\oplus d_\Qq^-d_\Qq \Qq^\ast$ an acyclic subcomplex of $(\Qq^\ast,d_\Qq)$.}
\end{proposition}

Note that  the Poincar\'e  DGCA  $(\Qq^*, {d_\Qq})$ in  Proposition \ref{prop:qsmall} is defined  uniquely by a  Hodge homotopy $d^-$ {for $(\Aa^\ast,d)$}. It is called the {\it small quotient algebra  of { $(\Aa^*,d,d^-)$}}  in\cite[Defintion 3.2]{FKLS2021}.

{
\begin{remark}Even when $(\Aa^\ast,d)$ is formal, one should not expect $\mathcal{L}_\Qq^\ast$ to be zero. That is, also for formal DGCAs admitting Hodge homotopies the small quotient algebra does not generally consist only of harmonic forms.
\end{remark}}

\subsection{The   unital $C_\infty$-algebra  associated with $\Qq^*$}\label{subs:cinftyq} 
Let $j\colon \Hh^ * \to \Qq^*$ be the  inclusion of the graded subspace of harmonic elements into $\Aa^\ast$.  Then  $j\colon  (\Hh^*, {d_\Hh} = 0) \to (\Qq^*, {d_\Qq})$ is a quasi-isomorphism of cochain complexes.
By definition of Hodge homotopy, $({d_\Qq^-}, \pi_{\Hh^*}, j)$  are homotopy data  of cochain  complexes, i.e.  we have the   following   commutative diagram 

\begin{eqnarray*}
	&\xymatrix{     *{ \quad \ \  \quad (\Qq^*, {d_\Qq})\ } \ar@(dl,ul)[]^{{d_\Qq^-}}\ \ar@<0.5ex>[r]^{\pi_{\Hh^*}} & *{\
			(\Hh^*, {d_\Hh}= 0).\quad \ \  \ \quad }  \ar@<0.5ex>[l]^{j}}
\end{eqnarray*}

The DGCA structure on $\Qq^*$ is a particular instance of a {\it unital  $C_\infty$-algebra} with unit $1\in \mathbb{K}=\Qq^0$, and the homotopy ${d_\Qq^-}$ satisfies 
side conditions \cite[Theorems 10, 12]{CG2008}
\begin{equation}\label{eq:side}
\pi_{\Hh^*}\circ {d_\Qq^-} = 0, \;   ({d_\Qq^-})^2 = 0,\:   {d_\Qq^-} (1) = 0. 
\end{equation}

  Thus we  can  transfer  the unital DGCA structure  of  $(\Qq^*, {d_\Qq})$ to a minimal unital $C_\infty$-structure on its  harmonic subspace  $\Hh^*$ \cite[Theorems 10, 12]{CG2008}.  
  Since this homotopy transfer is obtained by means of a Hodge homotopy, we will refer to it as a \emph{Hodge homotopy transfer.}
  The minimal unital $C^\infty$-structure  on $\Hh^*$ is defined  by  a sequence of  operations
 $m_k: \otimes ^k  \Hh^\ast \to \Hh^\ast[2-k]$, $k\geq 2$, whose unitality means that 
 \begin{equation}\label{eq:unital}
m_2 (1, a) =  m_2 (a, 1)  = a, \qquad \text  {  and } \qquad  m_n   (a_1, \ldots, 1, \ldots, a_n) = 0 \quad n\geq 3,
\end{equation}
and  which can be conveniently determined   
 by the  Kontsevich-Soilbeman  formula  for homotopy transfer \cite{KS2001} by \cite[Theorem 12]{CG2008}, cf. \cite[Lemma 48]{GCTV2012}, \cite[Theorem 5.2]{VdL2003}. 
 Using these formulas we obtain the following short proof of the main theorem from \cite{Zhou2019}, extending  our previous result with  Kawai and Schwachh\"ofer \cite[Theorem 4.1]{FKLS2021}, whose assertion is the next to  last statement of Theorem \ref{thm:A2-algebra}.

\begin{theorem}\label{thm:A2-algebra}
Let $r\geq 2$ and let $\Aa^\ast$ be a $(r-1)$-connected Poincar\'e DGCA of degree $n$ admitting a Hodge homotopy.
  If $n\leq \ell (r-1)+2$, with $\ell\geq 4$, then $\Hh^\ast$ carries a minimal unital $C_\infty$-algebra structure making it $C_\infty$-quasi-isomorphic to $\Aa^\ast$, whose multiplication $m_2$ is 
  $m_2(\alpha,\beta)=\pi_{\mathcal{H}^\ast}(\alpha\cdot\beta)$, and whose multiplications $m_k$ vanish for $k\geq \ell-1$. In particular, all the multiplications $m_k$ with $k\geq 5$ vanish if $n\leq 6r-4$ all the multiplications $m_k$ with $k\geq 4$ vanish if $n\leq 5r-3$. Also, all the multiplications $m_k$ with $k\geq 3$ vanish if $n\leq 4r-2$, and so  $\Aa^\ast$ is formal in this case.
  \end{theorem}
\begin{proof}
The DGCA $\Qq^\ast$ is  equivalent 
to $\Aa^*$. Hence, to  prove   Theorem \ref{thm:A2-algebra}  we  need only  to show that $\Hh^\ast$ carries a minimal unital $C_\infty$-algebra structure whose multiplication $m_2$ is 
  $\mu_2(\alpha,\beta)=\pi_{\mathcal{H}^\ast}(\alpha\cdot\beta)$ and whose multiplications $m_k$ vanish for $k\geq \ell -1$, making it quasi-isomorphic to $\Qq^\ast$. We can show this by means of the homotopy transfer theorem, using \cite[Theorem 10]{CG2008}. This endows $\Hh^\ast$ with an explicit unital minimal $C_\infty$-algebra structure making it $C_\infty$-quasi-isomorphic to  $\Qq^\ast$, and so to prove the statement in the theorem we only need to show that this unital minimal $C_\infty$-algebra structure actually has 
  $\mu_2(\alpha,\beta)=\pi_{\mathcal{H}^\ast}(\alpha\cdot\beta)$ and $m_k$ vanishing for $k\geq \ell-1$.

\par

 One has a convenient tree summation formula to express the higher multiplications $m_k$ obtained by homotopy transfer, see \cite[Theorem 12]{CG2008}, \cite{KS2001}. Namely, $m_k$ can be expressed as a sum over rooted trivalent trees with $k$ leaves.  Each tail edge of such
a tree is decorated by the inclusion $j\colon \Hh^\ast\hookrightarrow \Qq^\ast$, each internal edge is decorated by the operator ${d_\Qq^-}\colon \Qq^\ast\to \Qq^{\ast-1}$ and
the root edge is decorated by the operator $\pi_{\Hh^\ast}\colon  \Qq^\ast\to \Hh^\ast$; every
internal vertex is decorated by the multiplication $\cdot$ in $\Qq^\ast$.  The only graph with two leaves appearing in the tree summation formula is
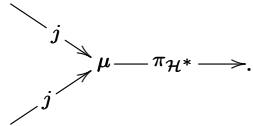
\[
\begin{xy}
,(-12,-8);(-7.2,-4.8)*{\,\scriptstyle{j}\,}**\dir{-}
,(-7.2,-4.8)*{\,\scriptstyle{j}\,};
(0,0)*{\,\,\scriptstyle{\mu}\,}**\dir{-}?>*\dir{>}
,(-12,8);(-6,4)*{\,\scriptstyle{j}\,}**\dir{-}
,(-6,4)*{\,\scriptstyle{j}\,};
(0,0)*{\,\,\scriptstyle{\mu}\,}**\dir{-}?>*\dir{>}
,(0,0)*{\,\,\scriptstyle{\mu}\,};
(9.6,0)*{\,\scriptstyle{\pi_{\Hh^\ast}}\,}**\dir{-}
,(9.6,0)*{\,\scriptstyle{\pi_{\Hh^\ast}}\,};
(19.2,0)**\dir{-}?>*\dir{>}
\end{xy}.
\]
This gives the announced result for $m_2$, so to conclude the proof we only need to show that $m_k$ vanishes for $k\geq \ell-1$.
Since the multiplications $m_k$ define a unital $C_\infty$-algebra structure, the element $
m_k(\alpha_1,\dots, \alpha_{k})$
vanishes if any of the homogeneous entries $\alpha_i$ has degree zero. We can therefore assume $\deg(\alpha_i)\geq r$ for any $i=1,\dots,k$. Then $m_k(\alpha_1,\dots, \alpha_{k})$ will have degree $h$ with $h\geq kr+(2-k)=k(r-1)+2$. If $k\geq \ell-1$ we have $h\geq (\ell-1)(r-1)+2=\ell(r-1)+2-(r-1)\geq n-r+1$. The only possibly nonzero result is then obtained with $h=n$. Let us compute
\[
\int m_k(\alpha_1,\dots, \alpha_{k}).
\] 
Since $k\geq \ell-1\geq 3$, the root of a tree $T$ contributing to $m_k$ will have a neighbourhood of one of the following forms:
\[
\begin{xy}
,(-12,-8);(-7.2,-4.8)*{\,\scriptstyle{j}\,}**\dir{-}
,(-7.2,-4.8)*{\,\scriptstyle{j}\,};
(0,0)*{\,\,\scriptstyle{\mu}\,}**\dir{-}?>*\dir{>}
,(-12,8);(-6,4)*{\,\scriptstyle{{d_\Qq^-}}\,}**\dir{-}
,(-6,4)*{\,\scriptstyle{{d_\Qq^-}}\,};
(0,0)*{\,\,\scriptstyle{\mu}\,}**\dir{-}?>*\dir{>}
,(0,0)*{\,\,\scriptstyle{\mu}\,};
(9.6,0)*{\,\scriptstyle{\pi_{\Hh^\ast}}\,}**\dir{-}
,(9.6,0)*{\,\scriptstyle{\pi_{\Hh^\ast}}\,};
(19.2,0)**\dir{-}?>*\dir{>}
\end{xy}\qquad;
\qquad
\begin{xy}
,(-12,-8);(-7.2,-4.8)*{\,\scriptstyle{{d_\Qq^-}}\,}**\dir{-}
,(-7.2,-4.8)*{\,\scriptstyle{{d_\Qq^-}}\,};
(0,0)*{\,\,\scriptstyle{\mu}\,}**\dir{-}?>*\dir{>}
,(-12,8);(-6,4)*{\,\scriptstyle{{d_\Qq^-}}\,}**\dir{-}
,(-6,4)*{\,\scriptstyle{{d_\Qq^-}}\,};
(0,0)*{\,\,\scriptstyle{\mu}\,}**\dir{-}?>*\dir{>}
,(0,0)*{\,\,\scriptstyle{\mu}\,};
(9.6,0)*{\,\scriptstyle{\pi_{\Hh^\ast}}\,}**\dir{-}
,(9.6,0)*{\,\scriptstyle{\pi_{\Hh^\ast}}\,};
(19.2,0)**\dir{-}?>*\dir{>}
\end{xy}.
\]
In degree $n$ all elements in $\Qq^\ast$ are harmonic, so the projection $\pi_{\Hh^\ast}$ is the identity in degree $n$. But then the orthogonality relations \eqref{eq:orthogonality} tell us that, if we denote by $m_T$ the contribution from the tree $T$ to $m_k$, then we have 
\[
\int m_T(\alpha_1,\dots,\alpha_k)=0.
\]
Since $\int\colon \Qq^{n}\to \mathbb{K}$ is an isomorphism, this gives $m_T(\alpha_1,\dots, \alpha_{k})=0$ for any tree $T$  contributing to $m_k$ and so $m_k(\alpha_1,\dots, \alpha_{k})=0$. 
  
\end{proof}

\begin{corollary}[Miller]\label{cor:miller} Let $r\geq 2$ and let $M$ be an $(r-1)$-connected compact manifold of dimension  less than or equal to $4r-2$. Then $M$ is formal.
\end{corollary}

\begin{remark}\label{rem:enrichment} (1)  The original  proof  by Miller   \cite{Miller1979} using  the Quillen theory  and the  other  two  proofs  of  Miller's results  using      Sullivan minimal  model   by  Fernandez-Munos \cite{FM2005} and Felix-Oprea-Tanr\'e \cite[Proposition 3.10, p. 110]{FOT2008}, as well as Zhou's proof of \cite[Theorem 3.2]{Zhou2019} are considerably  longer and less elementary than our proof of Theorem \ref{thm:A2-algebra}. This makes the proof of Theorem \ref{thm:A2-algebra} a good example of the effectiveness of the use of Hodge homotopies in real homotopy theory.
		
(2) Recall that we have a linear isomorphism $\mathcal{H}^\ast\to \ker({d_\Qq})\to H^\ast(\Qq)$ mapping each element in $\mathcal{H}^\ast$ to its cohomology class. One can use it to transfer the unital $C_\infty$-structure on $\mathcal{H}^\ast$ from Theorem \ref{thm:A2-algebra} to a unital $C_\infty$-structure on $H^\ast(\Qq)$. This transferred structure is a \emph{unital $C_\infty$-enrichment} of the DGCA-structure on $H^\ast(\Qq)$ induced by the multiplication $\cdot$ on $\Qq^\ast$. Indeed, let $\iota\colon H^\ast(\Qq)\to \mathcal{H}^\ast$ be the inverse of $[-]\colon \mathcal{H}^\ast\to  H^\ast(\Qq)$, and let $a,b$ be two cohomology classes in $H^\ast(\Qq)$. Let $\hat{a},\hat{b}\in \Qq^\ast$ be such that $[\hat{a}]=a$ and $[\hat{b}]=b$, and let $\alpha=\pi_{\mathcal{H}^\ast}(\hat{a})$ and $\beta=\pi_{\mathcal{H}^\ast}(\hat{b})$, respectively. Then
\[
\alpha=\hat{a} -{d_\Qq^-}{d_\Qq}\hat{a}-{d_\Qq}{d_\Qq^-}\hat{a}= \hat{a}-{d_\Qq}{d_\Qq^-}\hat{a},
\]
and so $[\alpha]=a$. Therefore $\iota(a)=\alpha$. Similarly, $\iota(b)=\beta$. Therefore, one finds
\begin{align*}
[m_2(\iota(a),\iota(b))]&=[m_2(\alpha,\beta)]=[\pi_{\mathcal{H}^\ast}(\alpha\cdot\beta)]\\
&=
[\alpha\cdot\beta-{d_\Qq^-}{d_\Qq}(\alpha\cdot\beta)-{d_\Qq}{d_\Qq^-}(\alpha\cdot\beta)]\\
&=[\alpha\cdot\beta]=[\alpha]\cdot[\beta]=a\cdot b.
\end{align*}
\end{remark}

\subsection{Zhou's conjecture and  two Cavalcanti's   formality  results}\label{subs:zhou}

In this subsection  we    prove  Theorem \ref{thm:zhou-s-conjecture}, which was conjectured by Zhou in \cite{Zhou2019}. This  theorem   is a generalization  of two Cavalcanti's    formality  results (Corollary \ref{prop:cavalcanti1})    and     a   variation  of  Theorem \ref{thm:A2-algebra}  in the  previous  subsection. 

 We begin with the following   remarks on  operations $m_k$  in  our minimal  unital $C_\infty$-algebra obtained from  a Hodge  homotopy   and the associated  small  quotient  algebra $(\Qq^*,  \cdot  , {d_\Qq}, {d_\Qq^-})$.

\begin{remark}
Let $r\geq 2$ and let $\Aa^\ast$ be a $(r-1)$-connected Poincar\'e DGCA of degree $n$ admitting a Hodge homotopy.
The explicit expression for $m_k$  are derived by using \cite[Theorem 3.4]{Merkulov1999}, \cite[Theorem 12]{CG2008}. One considers the recursive formula
\begin{align}\label{eq:hat-mk}
\notag \widehat{m}_2(\alpha_1,\alpha_2)&=\alpha_1\cdot\alpha_2,\\
\notag \widehat{m}_k (\alpha_1, &\dots, \alpha_k) =(-1)^{k-1}{d_\Qq^-}\widehat{m}_{k-1}(\alpha_1,\dots, \alpha_{k-1})\cdot\alpha_k\\
\notag & - (-1)^{k\deg(\alpha_1)} \alpha_1 \cdot {d_\Qq^-}\widehat{m}_{k-1}(\alpha_2,\cdots, \alpha_k)\\
&-\sum_{
i=2}^{k-2}
(-1)^{\nu}{d_\Qq^-}\widehat{m}_i(\alpha_1,\dots, \alpha_i) \cdot {d_\Qq^-}\widehat{m}_{k-i}(\alpha_{i+1},\dots, \alpha_k),
\end{align}
where { the $\alpha_i$'s are element in $\Hh^\ast\subseteq \Qq^\ast$ and $\cdot$ is the multiplication in the DGCA $(\Qq^\ast,d_{\Qq})$, and where} $\nu=i+(k-i-1)(\deg(\alpha_1)+\cdots+\deg(\alpha_i))$. Then the multiplication $m_k$ is defined by
\begin{equation}\label{eq:mk}
{m}_k (\alpha_1, \dots, \alpha_k)=\pi_{\mathcal{H}}\left(\widehat{m}_k (\alpha_1, \dots, \alpha_k)\right).
\end{equation}

One inductively sees that for any $k\geq 3$ and for any $\alpha\in \mathcal{H}^r$, one has
\begin{equation}\label{eq:hat-rrrrrrrr}
\widehat{m}_k(\alpha,\alpha,\dots,\alpha)=0.
\end{equation}
The base of the induction are given by $k=3$ and $k=4$.  For $k=3$, \eqref{eq:hat-mk} gives
\begin{equation}\label{eq:m3}
\widehat{m}_3 (\alpha, \beta, \gamma) =  {d_\Qq^-} (\alpha \cdot \beta) \cdot\gamma  - (-1)^{\deg  \alpha}  \alpha\cdot  {d_\Qq^-}  (\beta \cdot \gamma)
\end{equation}
and one immediately sees that $m_3(\alpha,\alpha,\alpha)=0$ both when $r$ is even and when $r$ is odd. For $k=4$, \eqref{eq:mk} together with the vanishing we just proved for $\widehat{m}_3(\alpha,\alpha,\alpha)$, gives
\[
\widehat{m}_4 (\alpha, \alpha, \alpha,\alpha)=-  {d_\Qq^-} (\alpha^2) \cdot  {d_\Qq^-} (\alpha^2)
\]
and this vanishes both when $\deg(\alpha)$ is odd, since $\alpha^2=0$ in this case, and when $\deg(\alpha)$ is even, since $\deg({d_\Qq^-}(\alpha^2))$ is odd in this case.
For an $k>4$ one has that both $k-1$ and at least one of the two indices $(i,j)$ subject to the constrain $i+j=k$ have to be greater or equal to 3, so the inductive assumption applies. As an immediate consequence, one has
\begin{equation}\label{eq:rrrrrrrr}
{m}_k(\alpha,\alpha,\dots,\alpha)=0
\end{equation}
for every $k\geq 3$.
\end{remark}

\begin{lemma}\label{lemma:rrr}
Let $r\geq 2$ and let $\Aa^\ast$ be a $(r-1)$-connected Poincar\'e DGCA of degree $n$ admitting a Hodge homotopy. If $n\leq\ell (r-1)+3$, then the multiplication $m_{\ell-1}(\alpha_1,\alpha_2,\dots,\alpha_{\ell-1})$ on $\mathcal{H}^\ast$ is zero unless $\deg(\alpha_1,\alpha_2,\dots,\alpha_{\ell-1})=(r,r,\dots,r)$. Moreover, for any $\alpha\in \mathcal{H}^r$ we have $m_{\ell-1}(\alpha,\alpha,\dots,\alpha)=0$. \end{lemma}
\begin{proof}
We know from Theorem \ref{thm:A2-algebra} that $m_{\ell -1}$ is zero if $n\leq\ell (r-1)+2$, so we may assume $n=\ell (r-1)+3$.
By unitality,
$m_{\ell-1}(\alpha_1,\alpha_2,\dots,\alpha_{\ell-1})=0$ if one among the $\deg(\alpha_i)$'s is zero, so we may assume $\deg(\alpha_i)\geq r$ for $i=1,\dots, \ell-1$. If $\deg(\alpha_{i_0})>r$ for some $i_0$, then $m_{\ell-1}(\alpha_1,\alpha_2,\dots,\alpha_{\ell-1})$ is an element in $\mathcal{H}^k$ with $k=\sum_i\deg(\alpha_i)-\ell+3>(\ell-1)r-\ell+3=n-r$, and so the only possibility for $m_{\ell-1}(\alpha_1,\alpha_2,\dots,\alpha_{\ell-1})$ to be nonzero is that $k=n$. But in this case $m_{\ell-1}(\alpha_1,\alpha_2,\dots,\alpha_{\ell-1})$ will vanish by the orthogonality relations. 
So we conclude that the only possibility for a nonzero $m_{\ell-1}(\alpha_1,\alpha_2,\dots,\alpha_{\ell-1})$ is  $\deg(\alpha_i)= r$ for every $i=1,\dots, \ell-1$. Finally, $m_{\ell-1}(\alpha,\alpha,\dots,\alpha)=0$ is equation \eqref{eq:rrrrrrrr}.
\end{proof}

\begin{lemma}\label{lemma:rrr2}
Let $r\geq 2$ and let $\Aa^\ast$ be a $(r-1)$-connected Poincar\'e DGCA of degree $n$ admitting a Hodge homotopy. If $n\leq\ell (r-1)+4$, then the multiplication $m_{\ell}(\alpha_1,\alpha_2,\dots,\alpha_{\ell})$ on $\mathcal{H}^\ast$ is zero unless $\deg(\alpha_1,\alpha_2,\dots,\alpha_{\ell})=(r,r,\dots,r)$. Moreover, for any $\alpha\in \mathcal{H}^r$ we have $m_{\ell}(\alpha,\alpha,\dots,\alpha)=0$. \end{lemma}
\begin{proof}
We know from Theorem \ref{thm:A2-algebra} that $m_\ell$ is zero if $n\leq (\ell+1)(r-1)+2$, and since $r\geq 2$ we have $\ell (r-1)+3\leq (\ell+1)(r-1)+2$. So we may assume $n=\ell (r-1)+4$.
By unitality, $m_{\ell}(\alpha_1,\alpha_2,\dots,\alpha_{\ell})=0$ if any of the $\deg(\alpha_i)$'s is zero, so we may assume $\deg(\alpha_i)\geq r$ for any $i=1,\dots,\ell$. If $\deg(\alpha_{i_0})>r$ for some $i_0$, then $m_{\ell}(\alpha_1,\alpha_2,\dots,\alpha_{\ell})\in \mathcal{H}^k$ with $k>\ell r+2-\ell\geq n-r$ and so the only possibility for $m_{\ell}(\alpha_1,\alpha_2,\dots,\alpha_{\ell})$ to be nonzero is that $k=n$. But in this case $m_{\ell}(\alpha_1,\alpha_2,\dots,\alpha_{\ell})$ will vanish by the orthogonality relations.   Finally, $m_{\ell}(\alpha,\alpha,\dots,\alpha)=0$ is equation \eqref{eq:rrrrrrrr}.

\end{proof}

\begin{lemma}\label{lemma:rrr-b}
Let $r\geq 2$ and let $\Aa^\ast$ be a $(r-1)$-connected Poincar\'e DGCA of degree $\ell (r-1)+4$ admitting a Hodge homotopy. The multiplication $m_{\ell-1}(\alpha_1,\dots,\alpha_{\ell-1})$ is zero unless  $\deg(\alpha_i)=r$ for all $i=1,\dots, \ell-1$ or $\deg(\alpha_i)=r$ for all $i=1,\dots,\ell-1$ except a single index $i_0$ for which one has $\deg(\alpha_{i_0})=r+1$.  Moreover, for any $\alpha\in \mathcal{H}^r$ we have $m_{\ell-1}(\alpha,\alpha,\dots,\alpha)=0$. 
\end{lemma}
\begin{proof}
By unitality, $m_{\ell-1}(\alpha_{1},\dots,\alpha_{\ell-1})=0$ if one among the $\deg(\alpha_i)$'s is zero, so we may assume $\deg(\alpha_i)\geq r$ for every $i=1,\dots,\ell-1$. If there is one index $i_0$ with $\deg(\alpha_{i_0})\geq r+2$ or there are two indices $i_0,i_1$ with $\deg(\alpha_{i_j})\geq r+1$, then $m_{\ell-1}(\alpha_1,\dots,\alpha_{\ell-1})$ is an element in $\mathcal{H}^k$ with $k=\sum_i\deg(\alpha_i)-\ell+1\geq (\ell-1)r+4-(\ell-1)>\ell(r-1)+4-r=n-r$, and so the only possibility for $m_{\ell-1}(\alpha_{1},\dots,\alpha_{\ell-1})=0$ to be nonzero is that $k=n$. In this case, $m_{\ell-1}(\alpha_{1},\dots,\alpha_{\ell-1})=0$ will vanish by the orthogonality relations.  Finally, $m_{\ell-1}(\alpha,\alpha,\dots,\alpha)=0$ is equation \eqref{eq:rrrrrrrr}.
 This completes  the  proof of Lemma \ref{lemma:rrr-b}.
\end{proof}

\begin{theorem}[Zhou's conjecture]\label{thm:zhou-s-conjecture}
Let $r\geq 2$ and let $\Aa^\ast$ be a $(r-1)$-connected Poincar\'e DGCA of degree $n$ admitting a Hodge homotopy.
  If $n\leq \ell (r-1)+4$, with $\ell\geq 4$ and $b_r=1$, then $\Hh^\ast$ carries a minimal unital $C_\infty$-algebra structure making it $C_\infty$-quasi-isomorphic to $\Aa^\ast$, whose multiplication $m_2$ is 
  $m_2(\alpha,\beta)=\pi_{\mathcal{H}^\ast}(\alpha\cdot\beta)$, and whose multiplications $m_k$ vanish for $k\geq \ell-1$. In particular, all the multiplications $m_k$ with $k\geq 5$ vanish if $n\leq 6r-2$ all the multiplications $m_k$ with $k\geq 4$ vanish if $n\leq 5r-1$. Also, all the multiplications $m_k$ with $k\geq 3$ vanish if $n\leq 4r$, and so  $\Aa^\ast$ is formal in this case.
\end{theorem}
\begin{proof}
If $n\leq \ell (r-1)+2$ this is Theorem \ref{thm:A2-algebra}. So there are only two cases to be considered: $n=\ell (r-1)+3$ and $\ell (r-1)+4$. 

\underline{Case 1}. If  $n=\ell (r-1)+3$, we proceed as follows.  Let $x$ be a linear generator of $\mathcal{H}^r$.
By Lemma \ref{lemma:rrr}, the only possible nonzero $m_{\ell-1}(\alpha_1,\dots,\alpha_{\ell-1})$'s are those with $\deg(\alpha_i)=r$ for $i=1,\dots,\ell-1$. So we are reduced to computing $m_{\ell-1}(x,x,\dots,x)$, and this is zero by Lemma \ref{lemma:rrr} again.  Finally, all the higher multiplications $m_k$ with $k\geq \ell$ vanish due to Theorem \ref{thm:A2-algebra}, since $\ell(r-1)+3\leq (\ell+1)(r-1)+2$ for $r\geq 2$.

\underline{Case 2}. If  $n=\ell (r-1)+4$, we proceed as follows. 
Let $x$ be a linear generator of $\mathcal{H}^r$.
By Lemma \ref{lemma:rrr-b}, the  multiplication $m_{\ell-1}(\alpha_1,\dots,\alpha_{\ell-1})$ is zero unless\\ 
(i)  $\deg(\alpha_i)=r$ for all $i=1,\dots, \ell-1$,\\
(ii) or  $\deg(\alpha_i)=r$ for all $i=1,\dots,\ell-1$ except  a single index $i_0$ for which one has $\deg(\alpha_{i_0})=r+1$.

\underline{Case 2,  step 1}. We start by showing that $m_{\ell-1}$ vanishes.

 In the first case (i) we are reduced to computing $m_{\ell-1}(x,x,\dots,x)$, and this vanishes by  equation \eqref{eq:rrrrrrrr}. 
 
 In the second case   (ii)  we are reduced to computing $m_{\ell-1}(x,\dots,x,v,x,\dots,x)$ with $v\in \mathcal{H}^{r+1}$, and we want to show that $m_{\ell-1}(x,\dots,x,v,x,\dots,x)=0$. Since 
$\deg(\widehat{m}_{\ell-1}(x,\dots,x,v,x,\dots,x))=\ell(r-1)+4-r=n-r$, from equation \eqref{eq:mk} we have $m_{\ell-1}(x,\dots,x,v,x,\dots,x)=\widehat{m}_{\ell-1}(x,\dots,x,v,x,\dots,x)$, so we want to show that $\widehat{m}_{\ell-1}(x,\dots,x,v,x,\dots,x)=0$.
By nondegeneracy of the pairing, this is equivalent to
\begin{equation*}
\int x \cdot \widehat{m}_{\ell-1}(x,\dots,x,v,x,\dots,x)=0.
\end{equation*}

Let us introduce the notation
\[
\{x\}_a=(\underbrace{x,x,\dots,x}_{a}); \qquad \{x,v,x\}_{a,b}=(\underbrace{x,\dots,x}_{a},v,\underbrace{x,\dots,x}_{b}),
\]
where $a$ and $b$ are possibly zero. With this notation, what we want to prove is
\[
\int x \cdot \widehat{m}_{\ell-1}(\{x,v,x\}_{a,\ell-2-a})=0
\]
for any $a=0,\dots,\ell-2$.

$\bullet$  We  provide a detailed proof in 
the generic case $2\leq a\leq \ell-4$.  Notice that we have $\ell\geq 6$ in this case. The remaining special cases with $a=0,1,\ell-3,\ell-2$ and/or with $\ell=4,5$ are handled similarly, so  we will be more sketchy.

 By \eqref{eq:hat-mk} and by the vanishing \eqref{eq:hat-rrrrrrrr}, we have
\begin{align*}
\widehat{m}_{\ell-1} (\{x,v,x\}_{a,\ell-2-a})& =(-1)^{\ell-2}{d_\Qq^-}\widehat{m}_{\ell-2}(\{x,v,x\}_{a,\ell-3-a})\cdot x\\
\notag & - (-1)^{(\ell-1)r} x \cdot {d_\Qq^-}\widehat{m}_{\ell-2}(\{x,v,x\}_{a-1,\ell-2-a})\\
&-\sum_{
i=2}^{a}
(-1)^{\nu}{d_\Qq^-}\widehat{m}_i(\{x\}_i) \cdot {d_\Qq^-}\widehat{m}_{\ell-1-i}(\{x,v,x\}_{a-i,\ell-2-a})\\
&-\sum_{
i=a+1}^{\ell-3}
(-1)^{\nu}{d_\Qq^-}\widehat{m}_i(\{x,v,x\}_{a,i-a-1}) \cdot {d_\Qq^-}\widehat{m}_{\ell-1-i}(\{x\}_{\ell-1-i})\\
& =(-1)^{\ell}{d_\Qq^-}\widehat{m}_{\ell-2}(\{x,v,x\}_{a,\ell-3-a})\cdot x\\
\notag & - (-1)^{(\ell-1)r} x \cdot {d_\Qq^-}\widehat{m}_{\ell-2}(\{x,v,x\}_{a-1,\ell-2-a})\\
&-
{d_\Qq^-}(x^2) \cdot {d_\Qq^-}\widehat{m}_{\ell-3}(\{x,v,x\}_{a-2,\ell-2-a})\\
&-
(-1)^{(\ell-1)r+\ell}{d_\Qq^-}\widehat{m}_{\ell-3}(\{x,v,x\}_{a,\ell-4-a}) \cdot {d_\Qq^-}(x^2).
\end{align*}
so we want to prove that
\begin{align*}
&(-1)^{\ell}\int x\cdot {d_\Qq^-}\widehat{m}_{\ell-2}(\{x,v,x\}_{a,\ell-3-a})\cdot x\\
 & - (-1)^{(\ell-1)r}\int  x^2 \cdot {d_\Qq^-}\widehat{m}_{\ell-2}(\{x,v,x\}_{a-1,\ell-2-a})\\
&-\int x\cdot
{d_\Qq^-}(x^2) \cdot {d_\Qq^-}\widehat{m}_{\ell-3}(\{x,v,x\}_{a-2,\ell-2-a})\\
&-
(-1)^{(\ell-1)r+\ell} \int x\cdot {d_\Qq^-}\widehat{m}_{\ell-3}(\{x,v,x\}_{a,\ell-4-a}) \cdot {d_\Qq^-}(x^2)=0.
\end{align*}
This is trivially satisfied if $r$ is odd, since in this case $x^2 = 0$. 

For an even $r$, taking into  account  that in this case  
$\deg ({d_\Qq^-}\widehat{m}_{\ell-3}(\{x,v,x\}_{a,\ell-4-a})) \equiv \ell+1 \mod 2$,  it reduces to
\begin{align}
\notag &(-1)^{\ell}\int x^2\cdot {d_\Qq^-}\widehat{m}_{\ell-2}(\{x,v,x\}_{a,\ell-3-a})\\
\notag & - \int  x^2 \cdot {d_\Qq^-}\widehat{m}_{\ell-2}(\{x,v,x\}_{a-1,\ell-2-a})\\
\notag&-\int x\cdot
{d_\Qq^-}(x^2) \cdot {d_\Qq^-}\widehat{m}_{\ell-3}(\{x,v,x\}_{a-2,\ell-2-a})\\
\label{eq:will-simplify}&+ \int x\cdot {d_\Qq^-}(x^2)\cdot {d_\Qq^-}\widehat{m}_{\ell-3}(\{x,v,x\}_{a,\ell-4-a}) =0.
\end{align}
We have
\[
x^2=\pi_\mathcal{H}(x^2)+{d_\Qq}{d_\Qq^-}(x^2)
\]
and
\[
\widehat{m}_{\ell-2}(\{x,v,x\}_{a,b})=\pi_{\mathcal{H}}\widehat{m}_{\ell-2}(\{x,v,x\}_{a,b})+{d_\Qq^-}{d_\Qq}\widehat{m}_{\ell-2}(\{x,v,x\}_{a,b})+
{d_\Qq}{d_\Qq^-}\widehat{m}_{\ell-2}(\{x,v,x\}_{a,b}).
\]
The orthogonality relations \eqref{eq:orthogonality}, the recursive formula \eqref{eq:hat-mk} and the vanishing \eqref{eq:hat-rrrrrrrr}, together with ${d_\Qq^-}(x^2)\cdot {d_\Qq^-}(x^2)=0$ since $\deg({d_\Qq^-}(x^2))$ is odd,  then give
\begin{align*}
(-1)^{\ell}\int x^2\cdot {d_\Qq^-}\widehat{m}_{\ell-2}(\{x,v,x\}_{a,\ell-3-a})&=(-1)^{\ell}\int {d_\Qq}{d_\Qq^-}(x^2)\, {d_\Qq^-}\widehat{m}_{\ell-2}(\{x,v,x\}_{a,\ell-3-a})\\
&=(-1)^{\ell}\int {d_\Qq^-}(x^2)\cdot {d_\Qq}{d_\Qq^-}\widehat{m}_{\ell-2}(\{x,v,x\}_{a,\ell-3-a})\\
&=(-1)^{\ell}\int {d_\Qq^-}(x^2)\cdot \widehat{m}_{\ell-2}(\{x,v,x\}_{a,\ell-3-a})\\
 & =  -\int x\cdot {d_\Qq^-}(x^2)\cdot {d_\Qq^-}\widehat{m}_{\ell-3}(\{x,v,x\}_{a,\ell-4-a})\\
 &\qquad  - (-1)^{\ell}\int x\cdot {d_\Qq^-}(x^2) \cdot {d_\Qq^-}\widehat{m}_{\ell-3}(\{x,v,x\}_{a-1,\ell-3-a})
\end{align*}
and, similarly,
\begin{align*}
-\int x^2\cdot {d_\Qq^-}\widehat{m}_{\ell-2}(\{x,v,x\}_{a-1,\ell-2-a})& =(-1)^\ell\int x\cdot {d_\Qq^-}(x^2)\cdot {d_\Qq^-}\widehat{m}_{\ell-3}(\{x,v,x\}_{a-1,\ell-3-a})\\
 &\qquad  +\int x\cdot {d_\Qq^-}(x^2) \cdot {d_\Qq^-}\widehat{m}_{\ell-3}(\{x,v,x\}_{a-2,\ell-2-a}).
\end{align*}
Substituting these into the left hand side of \eqref{eq:will-simplify} we see that \eqref{eq:will-simplify} is indeed satisfied.

$\bullet$  For $a=0$ and $\ell\geq 6$, we have to prove that
\begin{align*}
&(-1)^{\ell}\int x\cdot {d_\Qq^-}\widehat{m}_{\ell-2}(\{x,v,x\}_{0,\ell-3})\cdot x\\
&-
(-1)^{(\ell-1)r+\ell}\int x\cdot {d_\Qq^-}\widehat{m}_{\ell-3}(\{x,v,x\}_{0,\ell-4}) \cdot {d_\Qq^-}(x^2)=0.
\end{align*}
Again, this is trivially satisfied if $r$ is odd, while  for an even $r$  it is proved by the same argument used in the proof of \eqref{eq:will-simplify}.

The case $a=\ell-2$ and $\ell\geq 6$ is perfectly analogous.

$\bullet$  For $a=1$ and $\ell\geq 6$, we have to prove that

\begin{align*}
&(-1)^{\ell}\int x\cdot {d_\Qq^-}\widehat{m}_{\ell-2}(\{x,v,x\}_{1,\ell-4})\cdot x\\
\notag & - (-1)^{(\ell-1)r}\int x^2 \cdot {d_\Qq^-}\widehat{m}_{\ell-2}(\{x,v,x\}_{0,\ell-3})\\
&-
(-1)^{(\ell-1)r+\ell}\int x\cdot {d_\Qq^-}\widehat{m}_{\ell-3}(\{x,v,x\}_{1,\ell-5}) \cdot {d_\Qq^-}(x^2)=0,\\
\end{align*}
and again this is trivial for an odd $r$, while for an even $r$  it is proved by the same argument used in the proof of \eqref{eq:will-simplify}. The case $a=1$ and $\ell\geq 6$ is perfectly analogous.

$\bullet$  For $\ell=5$ we have four cases to consider; namely, $(v,x,x,x)$, $(x,v,x,x)$, $(x,x,v,x)$ and $(x,x,x,v)$. The fourth case is analogous to the first one, and the third one is analogous to the second one, so we need only investigate two cases. For the first one we have
\begin{align*}
\widehat{m}_{4} (v,x,x,x)& =-{d_\Qq^-}\widehat{m}_{3}(v,x,x)\cdot x+{d_\Qq^-}(vx) \cdot {d_\Qq^-}(x^2)
\end{align*}
so we want to prove
\[
-\int x\cdot {d_\Qq^-}\widehat{m}_{3}(v,x,x)\cdot x +\int x\cdot {d_\Qq^-}(vx) \cdot {d_\Qq^-}(x^2)=0.
\]
Again this is trivial for an odd $r$, while for an even $r$ we have
\begin{align*}
-\int x\cdot {d_\Qq^-}\widehat{m}_{3}(v,x,x)\cdot x&=-\int x^2\cdot {d_\Qq^-}\widehat{m}_{3}(v,x,x)\\
&=-\int {d_\Qq^-}(x^2)\cdot \widehat{m}_{3}(v,x,x)\\
&=-\int {d_\Qq^-}(x^2)\cdot \left({d_\Qq^-} (vx) \cdot x +  v\cdot  {d_\Qq^-}  (x^2)\right)\\
&=-\int {d_\Qq^-}(x^2)\cdot {d_\Qq^-} (vx) \cdot x  \\
&=-\int x\cdot {d_\Qq^-} (vx) \cdot {d_\Qq^-}(x^2). 
\end{align*}
For the second case we have
\begin{align*}
\widehat{m}_{4} (x,v,x,x)& =-{d_\Qq^-}\widehat{m}_{3}(x,v,x)\cdot x-x\cdot {d_\Qq^-}\widehat{m}_{3}(v,x,x) +{d_\Qq^-}(xv) \cdot {d_\Qq^-}(x^2)
\end{align*}
so we want to prove
\[
-\int x\cdot {d_\Qq^-}\widehat{m}_{3}(x,v,x)\cdot x-\int x^2\cdot {d_\Qq^-}\widehat{m}_{3}(v,x,x) +\int x\cdot {d_\Qq^-}(xv) \cdot {d_\Qq^-}(x^2)=0.
\]
Again this is trivial for an odd $r$, while for an even $r$ we have
\begin{align*}
-\int x\cdot {d_\Qq^-}\widehat{m}_{3}(x,v,x)\cdot x&=-\int x^2\cdot {d_\Qq^-}\widehat{m}_{3}(x,v,x)\\
&=-\int {d_\Qq^-}(x^2)\cdot \widehat{m}_{3}(x,v,x)\\
&=-\int {d_\Qq^-}(x^2)\cdot \left({d_\Qq^-} (xv) \cdot x  - x\cdot  {d_\Qq^-}  (vx)\right)\\
&=-0 
\end{align*}
and
\begin{align*}
-\int x^2\cdot {d_\Qq^-}\widehat{m}_{3}(v,x,x) &=-\int x\cdot {d_\Qq^-} (vx) \cdot {d_\Qq^-}(x^2)\\
&=-\int x\cdot {d_\Qq^-} (xv) \cdot {d_\Qq^-}(x^2)
\end{align*}
from the previous computation.

$\bullet$  For $\ell=4$ we have three cases to consider; namely, $(v,x,x)$, $(x,v,x)$, and $(x,x,v)$. The third case is analogous to the first one, so we need only investigate two cases. For the first one we have
\begin{align*}
\widehat{m}_3 (v,x,x) &=  {d_\Qq^-} (vx) \cdot x  - (-1)^{r+1}  v\cdot  {d_\Qq^-}  (x^2)\\
&={d_\Qq^-} (vx) \cdot x  -  {d_\Qq^-}  (x^2) \cdot v,
\end{align*}
so we want to prove
\[
\int x\cdot {d_\Qq^-} (vx) \cdot x  - \int x\cdot  {d_\Qq^-}  (x^2) \cdot v =0.
\]
Once more this is trivial for an odd $r$, while for an even $r$ we have
\begin{align*}
\int x\cdot {d_\Qq^-} (vx) \cdot x&=\int x^2\cdot {d_\Qq^-} (vx)\\
&=\int {d_\Qq^-}(x^2)\cdot {d_\Qq}{d_\Qq^-} (vx)\\
&=\int {d_\Qq^-}(x^2)\cdot vx\\
&=\int x\cdot {d_\Qq^-}(x^2)\cdot v.
\end{align*}
In the second case we have
\begin{align*}
\widehat{m}_3 (x,v,x) &=  {d_\Qq^-} (xv) \cdot x  - (-1)^{r}  x\cdot  {d_\Qq^-}  (vx).
\end{align*}
So this time the odd $r$ case is trivial by the usual reason, while also the even $r$ case is trivial since for an even $r$ we directly have $\widehat{m}_3 (x,v,x)=0$.

\

\underline{Case 2, final  step}. Now we turn to the $\ell$-ary multiplication $m_\ell$.  By Lemma \ref{lemma:rrr2}, since $\dim \mathcal{H}^r=1$ we are reduced to computing $m_\ell(x,x,\dots,x)$ and this is zero, again  by Lemma \ref{lemma:rrr2}.  Finally, all the higher multiplications $m_k$ with $k\geq \ell+1$ vanish due to Theorem \ref{thm:A2-algebra}, since $\ell(r-1)+4\leq (\ell+2)(r-1)+2$ for $r\geq 2$.

\end{proof}

The particular case $\ell=4$ considered in the last statement in Theorem \ref{thm:zhou-s-conjecture} recovers two results from \cite{Cavalcanti2006}, that we can jointly state as follows.

\begin{corollary}[Cavalcanti]\label{prop:cavalcanti1}   Let $r\geq 2$ and let $\Aa^\ast$ be a $(r-1)$-connected Poincar\'e DGCA of degree $n$ with $n\leq 4r$ admitting a Hodge homotopy. If $\dim H^r(\Aa)=1$ then $\Aa^\ast$ is formal. In particular an $(r-1)$-connected compact manifold  $M$  of  dimension  $n \leq 4r$    with $b_{r}  = 1$ is formal.
\end{corollary}

\section{Harrison cohomology and formality}\label{sec:harrform}

 Let $\Aa^*$ be an arbitrary  simply connected    Poincar\'e  DGCA endowed  with a Hodge homotopy. In this  section we  prove   several results   relating  the homotopy  type, i.e., the weak  equivalence  class, of  $\Aa^*$  to  the  Harrison  cohomology  class  $[\mu_3] \in {\rm HHarr} ^{3, -1} ( H^* (\Aa))$ associated  with the   minimal unital
 $C_\infty$-algebra  structure  on $H^* (\Aa)$ defined via the homotopy transfer we described  at the beginning of Subsection \ref{subs:cinftyq}. As shown in  Remark \ref{rem:enrichment}, such a minimal  unital $C_\infty$-structure  on $H^* (\Aa^*)$ is an example of a {\it minimal unital  $C_\infty$-enrichment of $(H^* (\Aa), \cdot)$}. In particular,  we show that $[\mu_3]$ is  the first  obstruction to the formality  of $\Aa^*$ (Lemma  \ref{lemma:first-obs-formality}),  a homotopy invariant  of  $\Aa^*$ (Lemma \ref{lemma:homotopy}),   and that it defines
  completely  the homotopy  type of a   $(r-1)$ connected  Poincar\'e  DGCA  of degree $n \le 4r -1$ admitting a Hodge homotopy  (Theorem \ref{thm:complete-invariant}).  Notice  several   results  in this section  are valid under   slightly weaker  conditions than, e.g., the   Hodge type  requirements,  which are  evident  from their proofs.

If $(N , m_1 =0, m_2, m_3, \ldots )$ is a minimal $A_\infty$-algebra, then $(N,m_2)$  is a graded associative algebra with  multiplication $ a\cdot  b = m_2 (a\otimes b)$. { To avoid confusion and at the same time cumbersome notation, in what follows we will write $N_{[\infty]}$ for the minimal $A_\infty$-algbera $(N , m_1 =0, m_2, m_3, \ldots )$ and simply $N$ for the graded associative algebra $(N,m_2)$.}
The  relation between the  homotopy  type  of the minimal $A_\infty$-algebra  
{$N_{[\infty]}$}
  and  the Hochschild  cochain complex ${\rm CHoch}^{*,*} (N, N)$ of the  graded associative  algebra 
 {$N$}, and in particular with the  Hochschild cohomology groups ${\rm HHoch} ^{n, 2-n} (N, N)$, 
 has  first been considered  by Kadeishvili \cite{Kadeishvili1988}, \cite{Kadeishvili2007}, see also \cite{Kadeishvili2023}.  Moving from $A_\infty$-algebras to $C_\infty$-algebras, any  result  concerning  the relation between minimal $A_\infty$-algebras and their Hochschild  cochain complex ${\rm CHoch} ^{*, *} (N, N)$  is also   valid for  minimal $C_\infty$-algebras,  replacing the Hochschild   cochain complex by the  Harrison cochain complex, i.e., by the subcomplex  of the  Hochschild  complex  ${\rm CHoch}^{*,*} (N, N)$  that consists of all cochains  vanishing on  shuffles. In particular,  Kadeishvili proved that   the homotopy type  of a minimal  $C_\infty$-algebra ${N_{[\infty]}=}(N, 0, m_2, m_3, \ldots)$ is   encoded  in  a suitable equivalence  class of   the Harrison  cochain $m = m_3 + m_4 + \ldots  \in   {\rm CHarr}^{*, 2-*} (N, N)$ { for the DGCA $N=(N,m_2)$}.  As  a  result,  he showed  that if ${\rm HHarr}^{n, 2-n}(N, N)  = 0$ for  $n \ge 3$  then    the  $C_\infty$-algebra   structure  on $N$ is formal, i.e., it is $C_\infty$-isomorphic to one with  $m _i = 0$ for all $i\ge 3$ \cite[Theorem 2]{Kadeishvili1988}. His theorem  implies   Tanre's result   on the formality  of a  simply connected  space  $M$ if  the  Harrison cohomology  ${\rm HHarr^{*, 2-*}} ({ H^*(M)}, { H^*(M)})$ vanishes  \cite{Tanre1985},  cf. \cite[Theorem 17]{Kadeishvili2023}.  We refer  the reader  to  Kadeishvili works \cite{Kadeishvili1988}, \cite{Kadeishvili2007},\cite{Kadeishvili2023}  for  a detailed      description  of  Harrison    cohomology   ${\rm HHarr}^{*, 2-*}(N, N)$
and the aforementioned  results due  to Kadeishvili.

The real homotopy type of a simply connected compact manifold $M$ is completely encoded in the homotopy class of its de Rham algebra of smooth forms \cite[Theorem D]{Sullivan1975}, see also \cite[Corollary 12.4, p. 117]{GM2013}. Since in characteristic zero the inclusion of DGCAs into $C_\infty$-algebras induces an equivalence of the homotopy categories \cite[Theorem 11.4.8, p.422]{loday-vallette}, we see that the real homotopy type of $M$ is completely encoded in the structure  of the minimal $C_\infty$-enhancement  of $(H^*_{\mathrm{dR}} (M), \wedge)$ 
induced on the de Rham cohomology $H^\ast_{\mathrm{dR}}(M)$ of $M$ by homotopy transfer via a Hodge homotopy as defined in Example \ref{ex:d-}.

Theorem \ref{thm:A2-algebra} and Remark \ref{rem:enrichment}  then tells us that if $r\geq 2$, then the real homotopy type of an $(r-1)$-connected compact $n$-dimensional manifold $M$ with $n\leq $$6r-4$ is completely encoded in the unital $C_\infty$-enrichment  $( H^\ast_{\mathrm{dR}}(M), \wedge, \mu_3,\mu_4)$ of the de Rham cohomology algebra $(H^\ast_{\mathrm{dR}}(M), \wedge)$ of $M$ by 
the ternary multiplication
\[
\mu_3\colon H^{k_1}_{\mathrm{dR}}(M)\otimes H^{k_2}_{\mathrm{dR}}(M)\otimes H^{k_3}_{\mathrm{dR}}(M) \to H^{k_1+k_2+k_3-1}_{\mathrm{dR}}(M)
\]
and the quaternary multiplication
\[
\mu_4\colon H^{k_1}_{\mathrm{dR}}(M)\otimes H^{k_2}_{\mathrm{dR}}(M)\otimes H^{k_3}_{\mathrm{dR}}(M)\otimes  H^{k_4}_{\mathrm{dR}}(M) \to H^{k_1+k_2+k_3+k_4-2}_{\mathrm{dR}}(M)
\]
induced by $m_3$ and $m_4$ via the canonical linear isomorphism $\mathcal{H}^\ast(M)\xrightarrow{\sim} H^\ast_{\mathrm{dR}}(M)$. Moreover, if $n\leq 5r-3$ then $\mu_4$ vanishes and so the
real homotopy type of $M$ is completely encoded in the unital $C_\infty$-enrichment $(\mathcal{H}^\ast_{\mathrm{dR}}(M), \wedge, \mu_3)$ of $(\mathcal{H}^\ast_{\mathrm{dR}}(M), \wedge)$. These 
facts can be seen as a unital $C_\infty$-algebra counterpart of results in rational homotopy theory by Crowley-Nordstr\"om \cite[Theorems 1,2]{CN} for $n=5r-3,$  and  they are closely  related  to   some result and conjecture by Nagy-Nordstr\"om \cite[Theorem 1.7 and Conjecture 1.9]{NN2021} for $n= 5r-2$.

Let now $\Aa^\ast$ be an arbitrary Poincar\'e DGCA of degree $n$ endowed with a Hodge homotopy. 
Recall that
the homotopy transfer works generally, and the result is a unital $C_\infty$-algebra $(\mathcal{H}^\ast, m_2, m_3,m_4,m_5,\dots)$, where the multiplications $m_2$, $m_3$ and $m_4$ are the same as in Theorem \ref{thm:A2-algebra}: none of the trees contributing to $m_2$, $m_3$ and $m_4$ is discarded there due to the bound on the dimension. Via the canonical linear isomorphism $\mathcal{H}^\ast\to H^\ast(\Aa)$ we get a unital $C_\infty$-enrichment $(H^\ast(\Aa), \cdot, \mu_3,\mu_4,\mu_5,\dots)$ of the cohomology algebra $(H^\ast(\Aa), \cdot)$ of $\Aa^\ast$. We focus on the ternary multiplication $\mu_3$. Since it is the ternary multiplication of a minimal $C_\infty$-algebra which is a unital $C_\infty$-enrichment  of $(H^\ast(\Aa),\cdot)$, it is a cocycle of degree $(3,-1)$ in the Harrison complex of $(H^\ast(\Aa), \cdot)$. 
The first relevant fact on this cocycle is the following.
\begin{lemma}\label{lemma:canonical}
The Harrison cohomology class of $\mu_3$ is independent of the Hodge homotopy used to define $m_3$. Therefore if $\Aa^\ast$ admits a Hodge homotopy we have a well defined cohomology class $[\mu_3]\in \mathrm{HHarr}^{3,-1}(H^\ast(\Aa),H^\ast(\Aa))$ only depending on the Poincar\'e DGCA $\Aa^\ast$ and not on the choice of a Hodge homotopy.
\end{lemma}
\begin{proof}
Let $(H^\ast(\Aa), \cdot, \mu_3,\mu_4,\mu_5,\dots)$ and $(H^\ast(\Aa), \cdot, \mu_3',\mu_4',\mu_5',\dots)$ two unital $C_\infty$-enrichments of $(H^\ast(\Aa), \cdot)$ induced by two Hodge homotopies. Then, by construction, the two unital $C_\infty$-algebras $(H^\ast(\Aa), \cdot, \mu_3,\mu_4,\mu_5,\dots)$ and $(H^\ast(\Aa), \cdot, \mu_3',\mu_4',\mu_5',\dots)$ are isomorphic by a $C_\infty$-isomorphism whose linear part is the identity of $H^\ast(\Aa)$. In this case, the quadratic componenst $\phi_2$ of this $C_\infty$-isomorphism is a Harrison cochain with $d_{\mathrm{Harr}}\phi_2=\mu_3=\mu_3'$, and so $[\mu_3]=[\mu_3']$.
\end{proof}

\begin{lemma}\label{lemma:first-obs-formality} The Harrison cohomology class $[\mu_3]$ is the first obstruction to the formality of a simply connected Poincar\'e DGCA $\Aa^\ast$. 
\end{lemma}
\begin{proof}
If $\Aa^\ast$ is formal, there exists an isomorphism of $C_\infty$-algebras between $(H^\ast(\Aa), \cdot, \mu_3,\mu_4,\mu_5,\dots)$ and $(H^\ast(\Aa), \cdot)$. The component $\phi_2$ of bidegree $(2,-1)$ of this $C_\infty$-isomorphism is a Harrison cochain with $d_{\mathrm{Harr}}\phi_2=\mu_3$. Hence if $\Aa^\ast$ is formal then $[\mu_3]=0$. 
\end{proof}
More generally one has the following.
\begin{lemma}\label{lemma:homotopy} The graded commutative algebra $(H^\ast(\Aa),\cdot)$ and the Harrison cohomology class $[\mu_3]$ are invariants of the DGCA homotopy class of a simply connected Poincar\'e DGCA $\Aa^\ast$ admitting a Hodge homotopy, i.e., if $\Aa^\ast$ and ${\Aa'}^\ast$ are in the same homotopy class, then there exists an isomorphism of graded commutative  algebras $\phi\colon  (H^\ast(\Aa),\cdot_{\Aa})\xrightarrow{\sim} (H^\ast(\Aa'),\cdot_{\Aa'})$ mapping $[\mu_3]$ to $[\mu_3']$ via the induced isomorphism of Harrison cohomologies.
\end{lemma}
\begin{proof}
If $\Aa^\ast$ and ${\Aa'}^\ast$ are in the same homotopy class then there exists a zig-zag of quasi-isomorphisms of DGCAs connecting $\Aa^\ast$ and ${\Aa'}^\ast$. So we are reduced to showing that if $\varphi\colon \Aa^\ast \to {\Aa'}^\ast$ is a quasi-isomorphism of DGCAs then there exists an isomorphism of graded commutative  $\phi\colon  (H^\ast(\Aa),\cdot_\Aa)\xrightarrow{\sim} (H^\ast(\Aa'),\cdot_{\Aa'})$ mapping $[\mu_3]$ to $[\mu_3']$ via the induced isomorphism of Harrison cohomologies. The same proof as for Lemma \ref{lemma:first-obs-formality} shows that $\phi=H^\ast(\varphi)$ is such an isomorphism.
\end{proof}

\begin{remark}\label{rem:only-first}
The vanishing of $[\mu_3]$ is the first obstruction to formality of $\Aa^\ast$, but it is generally not the only obstruction. Indeed, if $[\mu_3]=0$, then there exists a Harrison cochain $\phi_2$ of bidegree $(2,-1)$ with $d_{\mathrm{Harr}}\phi_2=\mu_3$ and one can use it to define a $C_\infty$-isomorphism between $(H^\ast(\Aa), \cdot, \mu_3,\mu_4,\mu_5,\dots)$ and a $C_\infty$-algebra of the form $(H^\ast(\Aa), \cdot, 0,\tilde{\mu}_4,\tilde{\mu}_5,\dots)$, were now the ternary multiplication vanishes. We then have the Harrison cohomology class of the new quaternary multiplication $\tilde{\mu}_4$, that is an element in $\mathrm{HHarr}^{4,-2}(H^\ast(\Aa),H^\ast(\Aa))$ that is the secondary obstruction to formality. Notice however that, as usual in obstruction theory, while the first obstruction $[\mu_3]$ is canonical, the higher obstructions are not. Indeed the Harrison cohomology class $[\tilde{\mu}_4]$ depends on the choice of the cobounding cochain $\phi_2$ for $\mu_3$. This fact can be seen as a Harrison cohomology counterpart that the Bianchi-Massey tensor defined in \cite{CN} is canonical, while the pentagonal tensor of \cite{NN2021} is not.
Similarly, the isomorphism class of the graded commutative algebra $(H^\ast(\Aa),\cdot)$ and the Harrison cohomology class $[\mu_3]$ are invariants of the homotopy class of $\Aa^\ast$, but generally they do not form a complete invariant.
\end{remark}

\begin{remark}
The Harrison cohomology class $[\mu_3]$, its canonicity, and its role as the first obstruction to the formality, as well as the corresponding Hochschild cohomology class for noncommutative DGAs are widely discussed in the literature. See \cite{BKS2003, JKM2022} and the references therein for additional information. 
\end{remark}

Despite we have observed in Remark \ref{rem:only-first} that the vanishing of $[\mu_3]$ is generally not a sufficient condition for the formality of $\Aa^\ast$ and that the pair $((H^\ast(\Aa),\cdot), [\mu_3])$ is not a complete invariant, there is at least an interesting situation where this invariant is complete so that in particular the condition $[\mu_3]=0$ is indeed sufficient for formality. This is the case when $\Aa^\ast$ is an $(r-1)$-connected Poincar\'e algebra of degree $4r-1$. In order to prove this result we need a couple of preliminary Lemmas.

Recall that the normalized Hochschild complex $\mathrm{CHoch}_\nu(A,A)$ of a graded associative $\mathbb{K}$-algebra $A$ with coefficients in itself is the bigraded subcomplex of the Hochschild complex $(\mathrm{CHoch}(A,A),d_{\mathrm{Hoch}})$ consisting on those cochains that vanish on constants, i.e., in each bidegree $(n,\bullet)$ of those $\varphi\colon A^{\otimes n}\to A$ such that $\varphi(a_1,\dots,a_n)=0$ if one of the $a_i$ is the unit of the algebra $A$ (or, equivalently, by $\mathbb{K}$-linearity, it is an element of $\mathbb{K}$). Note that operations $\mu_3$ and $\mu_4$ in  a unital minimal $C_\infty$-enrichment  of $(H^* (\Aa^*), \cdot)$ are  elements in $(\mathrm{CHoch}_\nu(H^* (\Aa^*),H^* (\Aa^*)),d_{\mathrm{Hoch}})$. The most important feature of the normalized Hochschild complex is the fact that the inclusion $\mathrm{CHoch}_\nu(A,A)\hookrightarrow \mathrm{CHoch}(A,A)$ is a quasiisomorphism, i.e., induces an isomorphism in cohomology (see, e.g., \cite[\S 1.5.7, p. 40]{loday}). However, in the next lemma we will not need this fact, but merely the fact that  $\mathrm{CHoch}_\nu(A,A)$ is a subcomplex of 
$\mathrm{CHoch}(A,A)$.

\begin{lemma}\label{lemma:prague}
Let $r\geq2$ and let $\Aa^\ast$ be an $(r-1)$-connected Poincar\'e algebra of degree $n$ with $n\leq 4r-1$. Let $(H^\ast(\Aa),\cdot,\mu_3)$ and $(H^\ast(\Aa),\cdot,\mu'_3)$ be two minimal $C_\infty$-enhancements of $(H^\ast(\Aa),\cdot)$ such that $\mu_3(a,b,c)$ and $\mu'_3(a,b,c)$ are zero unless $\deg(a,b,c)=(r,r,r)$. Then $[\mu_3]=[\mu_3']$ in $\mathrm{HHarr}^{3,-1}(H^\ast(\Aa),H^\ast(\Aa))$ if and only if there exists $\phi_2\in \mathrm{CHarr}^{2,-1}(H^\ast(\Aa),H^\ast(\Aa))$ such that
\begin{enumerate}
\item $d_{\mathrm{Hoch}}\phi_2=\mu_3-\mu_3'$;
\item $\phi_2(a,b)$ vanishes unless $\deg(a)$ and $\deg(b)$ are strictly positive integer multiples of $r$.
\end{enumerate}
\end{lemma}
\begin{proof}
The ``if" part is obvious, so let us prove the `only if' part. By assumption, there exists $\tilde{\phi}_2\in  \mathrm{CHarr}^{2,-1}_\nu(H^\ast(\Aa),H^\ast(\Aa))$ such that $d_{\mathrm{Harr}}\tilde{\phi}_2=\mu_3-\mu_3'$. Let us set
\[
\phi_2(a,b)=\begin{cases}
\tilde{\phi}_2(a,b)&\text{if $\deg(a)$ and $\deg(b)$ are strictly positive integer multiples of $r$}\\
0 &\text{otherwise}
\end{cases}
\]
We want to show that $d_{\mathrm{Harr}}\phi_2=\mu_3-\mu_3'$. If $\mathrm{min}(\deg(a),\deg(b),\deg(c))=0$, then by assumption, $\mu_3(a,b,c)=\mu_3'(a,b,c)=0$. So we need to show that $d_{\mathrm{Harr}}\phi_2(a,b,c)$ is zero in this case. Since the Harrison complex is a subcomsplex of the Hochschild complex and $\phi_2(a,b)$ vanishes when either $a$ or $b$ has degree zero, the Harrison cochain $\phi_2$ is an element in the normalized Hochschild complex of $H^\ast(\Aa)$. Since the normalized Hochschild complex is a subcomplex of the Hochschild complex we have 
\[
d_{\mathrm{Harr}}\phi_2=d_{\mathrm{Hoch}}\phi_2\in \mathrm{CHoch}_\nu(H^\ast(\Aa),H^\ast(\Aa)).
\]
Since $H^0(\Aa)=\mathbb{K}$, this tells us that also $d_{Harr}\phi_2$ vanishes when $\mathrm{min}(\deg(a),\deg(b),\deg(c))=0$. Now let us assume $\deg(a),\deg(b),\deg(c)>0$. 
If $\deg(a),\deg(b),\deg(c)$ are all multiples of $r$, then
\begin{align*}
d_{\mathrm{Harr}}\phi_2(a,b,c)&=
a\phi_2(b,c)-\phi_2(ab,c)+\phi_2(a,bc)-\phi_2(a,b)c\\
&=a\tilde{\phi}_2(b,c)-\tilde{\phi}_2(ab,c)+\tilde{\phi}_2(a,bc)-\tilde{\phi}_2(a,b)c\\
&=d_{\mathrm{Harr}}\tilde{\phi}_2(a,b,c)\\
&=\mu_3(a,b,c)-\mu_3'(a,b,c).
\end{align*}

If $\deg(a),\deg(b),\deg(c)$ are not all multiples of $r$, then by assumption we have that $\mu_3(a,b,c)=\mu_3'(a,b,c)=0$ so we have to show that also $d_{\mathrm{Harr}}\phi_2(a,b,c)$ vanishes in this case. Since we are assuming $a,b,c$ of strictly positive degree and $\Aa^\ast$ is $(r-1)$-connected this means we can assume $\deg(a),\deg(b),\deg(c)\geq r$. Let us write $\deg(a)=r+k_a$, and similarly for $b$ and $c$. Then the triple $(k_a,k_b,k_c)$ is a triple of nonnegative integers with $k_a+k_b+k_c\geq 1$. The element $d_{\mathrm{Harr}}\phi_2(a,b,c)$ is an element in $H^k(\Aa^\ast)$ with $k=3r-1+(k_a+k_b+k_c)\geq 3r\geq n-r+1$. Since $\Aa^\ast$ is an $(r-1)$-connected, this element will automatically be zero unless $k=n$, i.e., unless
$k_a+k_b+k_c=r$.

Since $k_a,k_b,k_c$ are nonnegative integers, if they are all nonzero, then the condition $k_a+k_b+k_c=r$ implies $0<k_a,k_b,k_c<r$, and so 
\[
r<\deg(a), \deg(b),\deg(c)<2r<\deg(ab), \deg(bc).
\]
This implies $d_{\mathrm{Harr}}\phi_2(a,b,c)=0$. To analyze the remaining cases, let us begin by assuming $k_a=0$. If $0<k_b<r$  (or, equivalently, $0<k_c<r$) then we have
\[
r<\deg(b),\deg(c)<2r<\deg(ab), \deg(bc),
\]
so again $d_{\mathrm{Harr}}\phi_2(a,b,c)=0$. The cases $k_b=0$ and $0<k_a<r$ (or, equivalently, $0<k_c<r$), or $k-c=0$ and $0<k_a<r$ (or, equivalently, $0<k_b<r$) are analogous.
We are therefore left with only three cases to consider:
\[
(k_a,k_b,k_c)\in \{(0,0,r), (0,r,0),(r,0,0)\}.
\]
But in all these three cases we have that $\deg(a),\deg(b),\deg(c)$ are all multiples of $r$, against our assumption.
\end{proof}

\begin{remark}
In the proof of Lemma \ref{lemma:prague} one may be tempted to avoid going through the normalized Hochschild complex and work directly with the normalized Harrison complex. Unfortunately the definition of the latter is not straightforward in general, see \cite[Section 3.2]{graves}, so we preferred to go through the simpler Hochschild framework. 
\end{remark}

\begin{corollary}\label{cor:prague}
Let $r\geq2$ and let $\Aa^\ast$ be an $(r-1)$-connected Poincar\'e algebra of degree $n$ with $n\leq 4r-1$. Let $(H^\ast(\Aa),\cdot,\mu_3)$ and $(H^\ast(\Aa),\cdot,\mu'_3)$ be two minimal unital $C_\infty$-enrichments  of $(H^\ast(\Aa),\cdot)$ such that $\mu_3(a,b,c)$ and $\mu'_3(a,b,c)$ are zero unless $\deg(a,b,c)=(r,r,r)$. Then $(H^\ast(\Aa),\cdot,\mu_3)$ and $(H^\ast(\Aa),\cdot,\mu'_3)$ are isomorphic  $C_\infty$-algebras via a  $C_\infty$-isomorphisms whose linear  component is the identity if and only if $[\mu_3]=[\mu_3']$ in $\mathrm{HHarr}^{3,-1}(H^\ast(\Aa),H^\ast(\Aa))$.
\end{corollary}
\begin{proof}
The ``only if" part is obvious, so we only prove the ``if" part. Assume $[\mu_3]=[\mu_3']$ in $\mathrm{HHarr}^{3,-1}(H^\ast(\Aa),H^\ast(\Aa))$. Then by Lemma \ref{lemma:prague} there exists $\phi_2$ in $\mathrm{CHarr}^{2,-1}(H^\ast(\Aa),H^\ast(\Aa))$ such that $d_{\mathrm{Harr}}\phi_2=\mu_3-\mu_3'$ and with
$\phi_2(a,b)=0$ unless $\deg(a)$ and $\deg(b)$ are strictly positive integer multiples of $r$. We claim that $(\mathrm{id}_{H^\ast(\Aa)},\phi_2)$ is an isomorphism of $C_\infty$-algebras between $(H^\ast(\Aa),\cdot,\mu_3)$ and $(H^\ast(\Aa),\cdot,\mu'_3)$. Indeed, since we only have multiplications $\mu_2=\mu_2'=\cdot,\mu_3,\mu_3'$ and components $\mathrm{id}_{H^\ast(\Aa)}$ and $\phi_2$ for the morphism, the $C_\infty$-morphism equations 
\begin{align}
 \sum_{k =1} ^p\sum_{r_1+ \ldots r_k = n}  (-1)^\eta \mu_k '  (\phi_{r_1} (a_1, \ldots, a_{r_1}), \ldots, \phi_{r_k}(a_{n-r_k +1}, \ldots, a_p)) \nonumber\\
 = \sum_{k =1} ^p \sum_{\lambda= 0} ^{p-k} (-1)^\xi \phi_{p-k+1}  (a_1, \ldots, a_\lambda, m_k (a_{\lambda+1}, \ldots, a_{\lambda+k}),a_{\lambda+k +1}, \ldots, a_p )\label{eq:amorphism}
\end{align}
where 
$$\eta = \sum_{1\le \alpha < \beta \le k}  (|a_{{r_1} + \ldots + r_{\alpha -1} +1}|  + \ldots + |a_{{r_1} + \ldots + _{r_\alpha}}|+ r_\alpha)(1+r_\beta),  $$
$$\xi = k  + k\lambda + k ( |a_1| +\ldots + |a_\lambda|) +p. $$
are trivial for $p\geq 7$. indeed, since $\phi_n=0$ for $n>2$ and $\mu_n,\mu'_n=0$ for $n>3$, there is no way of getting a nontrivial term with 7 or more arguments of the form $\mu'_n(\phi_{i_1},\phi_{i_2},\ldots,\phi_{i_n})$ or $\phi_n(\ldots,\mu_k,\ldots)$ under these conditions: the most one can get in the first case is $\mu_3'(\phi_2,\phi_2,\phi_2)$ and in the second case $\phi_2(\mu_3,\mathrm{id}_H^\ast(\Aa),\mathrm{id}_H^\ast(\Aa))$, or $\phi_2(\mathrm{id}_H^\ast(\Aa),\mu_3,\mathrm{id}_H^\ast(\Aa))$, or $\phi_2(\mathrm{id}_H^\ast(\Aa),\mathrm{id}_H^\ast(\Aa),\mu_3)$. So we only have to check equations  \eqref{eq:amorphism} for $p\leq 6$. Equations \eqref{eq:amorphism} for $p\leq 2$ are trivially satisfied, while for $p=3$ we precisely have the equation $d_{\mathrm{Harr}}\phi_2=\mu_3-\mu_3'$. So we are left with the three equations corresponding to $p=4,5,6$.
For $p=4$ we have the equation
\begin{align}\nonumber
\pm \phi_{2}& (a_1, a_{2})\cdot  \phi_{2}(a_3, a_4)) 
\pm \mu_3 '  (\phi_2 (a_1, a_{2}),a_3, a_4)) 
\pm \mu_3 '  (a_1,\phi_2(a_2,a_3), a_4)) \\
\nonumber&\qquad\qquad\pm \mu_3 '  (a_1,a_2,\phi_2(a_3,a_4)) \\
& =  \pm\phi_{2}  (\mu_3(a_1, a_2,a_3), a_4 )\pm \phi_{2}  (a_1, \mu_3(a_2, a_3,a_4)), \label{eq:p=4}
\end{align}
where we do not make the signs explicit since they are inessential here.
If any of the $a_i$ has degree 0, then both the left hand side and the right hand side are zero. Since $\Aa^\ast$ is  $(r-1)$-connected we can therefore assume that the cohomology classes $a_i$ have degree at least $r$. The element $\mu_3(a_i, a_j,a_k)$, if nonzero, is in degree $3r-1$; therefore, the right hand side of \eqref{eq:p=4} is always zero. Similarly, the element $\phi_2(a_i, a_j)$, if nonzero, is in degree $k$ with $k\equiv -1 \mod r$; therefore, the only possibly nonzero term in the left hand side of \eqref{eq:p=4} is the first one, and so \eqref{eq:p=4} reduces to
\begin{equation}\label{eq:p4}
\phi_{2} (a_1, a_{2})\cdot  \phi_{2}(a_3, a_4)=0.
\end{equation}
This is automatically satisfied unless $\deg(a_i)$ is a strictly positive multiple of $r$, for any $i=1,\dots,4$. In this case, the left hand side of \eqref{eq:p4} is an element in $H^{kr-2}(\Aa^\ast)$ for some $k\geq 4$. Since $\Aa^\ast$ is an $(r-1)$-connected Poncar\'e algebra of degree $4r-1$ with $r\geq 2$ we have that $H^{kr-2}(\Aa^\ast)=0$ for any $k\geq 4$, and so also in this case equation \eqref{eq:p4} is automatically satisfied.

 For $p=5$ we have the equation
\begin{align*}
\pm \mu_3 ' & (\phi_2 (a_1, a_{2}),\phi_2(a_3, a_4),a_5) 
\pm \mu_3 '  (\phi_2(a_1,a_2),a_3,\phi_2(a_4,a_5)) \\
&\pm \mu_3 '  (a_1,\phi_2(a_2,a_3),\phi_2(a_4,a_5)) 
=0.
\end{align*}
Since $\phi_2(a_i,a_j)$ can never be a nonzero element of degree $r$, this equation is automatically satisfied. Finally for $p=6$ we have the equation
\begin{align*}
\pm \mu_3 '  (\phi_2 (a_1, a_{2}),\phi_2(a_3, a_4),\phi_2(a_5,a_6))=0 
\end{align*}
and this is automatically satisfied by the same argument as for the $p=5$ case.
\end{proof}

Putting together Lemma \ref{lemma:first-obs-formality} and Corollary \ref{cor:prague} we obtain the following.

\begin{theorem}\label{thm:complete-invariant} Let $r\geq 2$. The graded commutative algebra $(H^\ast(\Aa),\cdot)$ and the Harrison cohomology class $[\mu_3]$ are complete invariants of the DGCA homotopy class of an $(r-1)$ connected Poincar\'e DGCA $\Aa^\ast$ of degree $n$ with $n\leq 4r-1$ admitting a Hodge homotopy, i.e., $\Aa^\ast$ and ${\Aa'}^\ast$ are in the same homotopy class if and only if there exists an isomorphism of graded commutative algebras $\phi\colon  (H^\ast(\Aa),\cdot_\Aa)\xrightarrow{\sim} (H^\ast(\Aa'),\cdot_{\Aa'})$ mapping $[\mu_3]$ to $[\mu_3']$ via the induced isomorphism of Harrison cohomologies. In particular, $\Aa^\ast$ is formal if and only if $[\mu_3]=0$.
\end{theorem}
\begin{proof}The `only if' part is Lemma \ref{lemma:first-obs-formality}, so we only need to prove the `if' part. Assume there is an isomorphism $\phi$ as in the statement. Then $\phi$ is a linear $C_\infty$-isomorphism between $(H^\ast(\Aa),\cdot_{\Aa},\mu_3)$ and the $C_\infty$-algebra $(H^\ast(\Aa'),\cdot_{\Aa'},\mu^\phi_3)$, where $\mu_3^\phi$ is the degree $-1$ ternary multiplication
\[
\phi\circ \mu_3\circ (\phi^{-1}\otimes \phi^{-1}\otimes \phi^{-1})\colon H^\ast(\Aa')^{\otimes 3} \to H^\ast(\Aa')
\]
on $H^\ast(\Aa')$. By assumption, $[\mu_3^\phi]=[\mu_3']$. Moreover, the $C_\infty$-algebras $H^\ast(\Aa'),\cdot_{\Aa'},\mu_3^\phi)$ and $H^\ast(\Aa'),\cdot_{\Aa'},\mu_3')$ satisfy the assumptions of  Corollary \ref{cor:prague}. Therefore $H^\ast(\Aa'),\cdot_{\Aa'},\mu_3^\phi)$ and $H^\ast(\Aa'),\cdot_{\Aa'},\mu_3')$ are isomorphic $C_\infty$-algebras and so also $H^\ast(\Aa),\cdot_{\Aa},\mu_3)$ and $H^\ast(\Aa'),\cdot_{\Aa'},\mu_3')$ are isomorphic.

\end{proof}

\begin{corollary}\label{cor:complete-invariant}
Let $r\geq 2$. 
The real homotopy type of  an $(r-1)$ connected compact manifold $M$ of dimension $n$ with $n\leq 4r-1$ is completely encoded in the de Rham cohomology algebra $(H^\ast_{\mathrm{dR}}(M),\wedge)$ of $M$ and in the  Harrison cohomology class $[\mu_3]\in \mathrm{HHarr}^{3,-1}(H^\ast_{\mathrm{dR}}(M),H^\ast_{\mathrm{dR}}(M))$.
\end{corollary}

\begin{remark}
If $n\leq 4r-2$ then $\mu_3=0$ and $\Aa^\ast$ is formal, see Theorem \ref{thm:A2-algebra}. So Theorem \ref{thm:complete-invariant} is actually a statement on $(r-1)$ connected Poincar\'e DGCAs of degree $4r-1$. 
\end{remark}

\section{Harrison cohomology, symmetric cohomology and the Bianchi-Massey tensor}\label{sec:HarrBM}
In this  section  we  relate  the Harrison  cohomology  class  $[\mu_3]$ with  the  Bianchi-Massey tensor   defined  by Crowley-Nordstr\"om in \cite{CN}  (Theorem \ref{thm:HarrBM}.)
The proof  of Theorem \ref{thm:HarrBM} is based on the spectral sequence symmetric-to-Harrison cohomology from \cite{Barr1968,Fleury1973}, giving  in characteristic zero
a canonical isomorphism
\[
H_{\mathrm{Sym}}^{\bullet-1}(H^\ast(\Aa),H^\ast(\Aa))\cong \mathrm{HHarr}^\bullet(H^\ast(\Aa),H^\ast(\Aa))
\]
relates the Harrison cohomology class $[\mu_3]$ to the Bianchi-Massey tensor from \cite{CN}. Let us first recall a few general features of the symmetric-to-Harrison cohomology spectral sequence. For a graded commutative algebra $B$ and a graded $B$-module $N$, the 0-th page of the spectral sequence is 
 \[
 E_0^{p,q}=\mathrm{CHarr}^{p+1}(S^{q+1}B,N),
 \]
 where $S$ is the (graded) symmetric algebra functor, i.e.,
 \[
  SV =\bigoplus_{k=0}^\infty \mathrm{Sym}^k(V)
 \]
 for any (graded) vector space $V$, and $S^n$ will denote the $n$-fold composition of $S$ with itself: it should not be confused with the $n$-th (graded) symmetric power that, in the hope of avoiding confusion, we are denoting with $\mathrm{Sym}^n$.  The horizontal differential
 \[
 d_{\mathrm{Harr}}\colon \mathrm{CHarr}^{p+1}(S^{q+1}B,N)\to \mathrm{CHarr}^{p+2}(S^{q+1}B,N)
 \]
 is the Harrison differential, while the vertical differential
 \[
  d_{\mathrm{Sym}}\colon \mathrm{CHarr}^{p+1}(S^{q+1}B,N)\to \mathrm{CHarr}^{p+1}(S^{q+2}B,N)
 \]
 is induced by the multiplication map from the symmetric algebra over a (graded) commutative algebra to the algebra itself. More precisely, the multiplication maps induce maps 
 \[
 \partial_i\colon S^{q+2}B=S^{q+1-i}(S(S^{i}B))\to S^{q+1-i}(S^{i}B)=S^{q+1}B,
 \]
 and the alternate sum $d_{S^\bullet B}=\sum_{i=0}^{q+1}(-1)^i\partial_i$ produces a chain complex
 \[
 \cdots \to S^{q+2}B\xrightarrow{d_{S^\bullet B}} S^{q+1}B \xrightarrow{d_{S^\bullet B}} S^{q}B \to \cdots
 \]
By applying $\mathrm{CHarr}^{p+1}(-,N)$ we get the vertical differential $d_{\mathrm{Sym}}$. The first page of the spectral sequence obtained by starting with the horizontal differential has
 \[
 ({}'E)^{p,q}_1=\mathrm{HHarr}^{p+1}(S^{q+1}B),N)
 \]
 By \cite[Proposition 3.1]{Barr1968} one has 
 \[
 \mathrm{HHarr}^{k}(SV,M)=
  \begin{cases}
\mathrm{Hom}(V,M)&\text{if $k=1$};
\\
  0 &\text{if $k>1$}
\end{cases}
\]
 for every finitely dimensional graded vector space $V$ and every graded $SV$-module $M$, for any $k>1$. Therefore, for any $q\geq 0$
  \[
{}'E^{p,q}_1=
  \begin{cases}
\mathrm{Hom}(S^qB,N)&\text{if $p=0$};
\\
  0 &\text{if $p>0$}
\end{cases}
\]
and the page ${}'E_1$ takes the form

 \begin{tikzpicture}
  \matrix (m) [matrix of math nodes,
    nodes in empty cells,nodes={minimum width=5ex,
    minimum height=5ex,outer sep=-5pt},
    column sep=1ex,row sep=1ex]{
                &      &     &     & \\
          2     &  \mathrm{Hom}(S^2B,N) &  0  & 0 & \\     
          1     &  \mathrm{Hom}(SB,N) &  0  &0 & \\
          0     &  \mathrm{Hom}(B,N)  & 0 &  0  & \\
    \quad\strut &   0  &  1  &  2  & \strut \\};
  \draw[-stealth] (m-4-2.north) -- (m-3-2.south);
   \draw[-stealth] (m-3-2.north) -- (m-2-2.south);
    \draw[-stealth] (m-2-2.north) -- (m-1-2.south);
\draw[thick] (m-1-1.east) -- (m-5-1.east) ;
\draw[thick] (m-5-1.north) -- (m-5-5.north) ;
\end{tikzpicture} 

Therefore ${}'E_2={}'E_\infty$ is 

 \begin{tikzpicture}
  \matrix (m) [matrix of math nodes,
    nodes in empty cells,nodes={minimum width=5ex,
    minimum height=5ex,outer sep=-5pt},
    column sep=1ex,row sep=1ex]{
                &      &     &     & \\
          2     &  H^2_{\mathrm{Sym}}(B,N) &  0  & 0 & \\     
          1     &  H^1_{\mathrm{Sym}}(B,N) &  0  &0 & \\
          0     &  H^0_{\mathrm{Sym}}(B,N)  & 0 &  0  & \\
    \quad\strut &   0  &  1  &  2  & \strut \\};
\draw[thick] (m-1-1.east) -- (m-5-1.east) ;
\draw[thick] (m-5-1.north) -- (m-5-5.north) ;
\end{tikzpicture}

Now let us consider the spectral sequence starting with the vertical differential. Barr shows that ${}''E_1$ has the form

  \begin{tikzpicture}
  \matrix (m) [matrix of math nodes,
    nodes in empty cells,nodes={minimum width=5ex,
    minimum height=5ex,outer sep=-5pt},
    column sep=1ex,row sep=1ex]{
                &      &     &     & &\\
          2     &  0 &  0  & 0 & &\\     
          1     &  0 &  0  &0 & &\\
          0     &  \mathrm{CHarr}^{1}(B,N)\,  &\, {\mathrm{CHarr}}^{2}(B,N)\, &  \,{\mathrm{CHarr}}^{3}(B,N)\,  & \phantom{0}& \\
    \quad\strut &   0  &  1  &  2  & \strut \\};
  \draw[-stealth] (m-4-2.east) -- (m-4-3.west);
   \draw[-stealth] (m-4-3.east) -- (m-4-4.west);
    \draw[-stealth] (m-4-4.east) -- (m-4-5.west);
\draw[thick] (m-1-1.east) -- (m-5-1.east) ;
\draw[thick] (m-5-1.north) -- (m-5-5.north) ;
\end{tikzpicture} 

Therefore ${}''E_2={}''E_\infty$ is

 \begin{tikzpicture}
  \matrix (m) [matrix of math nodes,
    nodes in empty cells,nodes={minimum width=5ex,
    minimum height=5ex,outer sep=-5pt},
    column sep=1ex,row sep=1ex]{
                &      &     &     & &\\
          2     &  0 &  0  & 0 & &\\     
          1     &  0 &  0  &0 & &\\
          0     &  \mathrm{HHarr}^1(B,N)\,  &\,\mathrm{HHarr}^2(B,N)\, &  \,\mathrm{HHarr}^3(B,N)\,  & \phantom{0}& \\
    \quad\strut &   0  &  1  &  2  & \strut \\};
\draw[thick] (m-1-1.east) -- (m-5-1.east) ;
\draw[thick] (m-5-1.north) -- (m-5-5.north) ;
\end{tikzpicture} 

\noindent and by comparing the two $E_\infty$ pages one finds the isomorphism $H_{\mathrm{Sym}}^{\bullet-1}(B,N)\cong \mathrm{HHarr}^\bullet(B,N)$. Moreover, 
the usual zig-zag argument then tells us that an element $[\xi]$ in $\mathrm{Harr}^3(B,N)={}''E^{2,0}_\infty$, that is naturally represented by an element in $E^{2,0}_0$, is also represented by an element $\Xi$ in $E^{1,1}_0$. The element $\Xi$ is an element in $\mathrm{CHarr}^{2}(S^{2}B,N)$ and so a morphism
\[
\Xi\colon \mathrm{Sym}^2(S^2B)\to N
\]
characterized by the fact that 
\[
d_{\mathrm{Harr}}\Xi=d_{\mathrm{Sym}}\xi
\]
Here we can take as $\xi$ a representative for $[\xi]$ in ${}''E^{2,0}_1=\mathrm{CHar}^3(B,N)$ and look at it as a representative for $[\xi]$ in $E^{2,0}_0=\mathrm{CHar}^3(SB,N)$ via the multiplication map $\epsilon\colon SB\to B$. The element $d_{\mathrm{Sym}}\xi$ is then the further pullback to $\mathrm{CHarr}^3(S^2B,N)$ via the multiplication map $\epsilon\colon S^2B\to SB$ Since the algebras $S^nB$ are iteratively generated by $B$, 
the equation $d_{\mathrm{Harr}}\Xi=d_{\mathrm{Sym}}\xi$ reduces to
\begin{equation}\label{eq:Harr-to-sym}
-x \Xi(y,z)+ \Xi(x\odot y,z)- \Xi(x,y\odot z)+ \Xi(x,y)z=\xi(x,y,z)
\end{equation}
for any $x,y,z$ in $B$. We can now prove the main statement of this section. First we need recalling the definition of the Bianchi-Massey tensor of a Poincar\'e DGCA $\Aa^\ast$ endowed with a Hodge homotopy, see \cite{CN}.
\begin{definition}\label{def:BM}
Let $(\Aa^\ast,d,d^-)$ be a Poincar\'e DGCA endowed with a Hodge homotopy. Let $\iota\colon H^\ast(\Aa)\to \Aa^\ast$ be the composition $H^\ast(\Aa)\xrightarrow{\sim} \mathcal{H}^\ast\hookrightarrow \Aa^\ast$, where the first isomorphism is the inverse to the projection $\mathcal{H}^\ast\to H^\ast(\Aa)$ and the second arrow is the inclusion of $\mathcal{H}^\ast$ in $\Aa^\ast$. Denote by the same symbol $\iota$ the extension of $\iota$ to a morphism of algebras $S^n(H^\ast(\Aa))\to S^n(\Aa^\ast)$ for any $n\geq 1$, and let $\epsilon\colon S^n\Aa^\ast\to \Aa^\ast$ be the iterated multiplication, for any $n\geq 1$. Finally, let 
\[
\alpha_2\colon \mathrm{Sym}^2(H^\ast(\Aa))\to \Aa^\ast
\]
be the restriction of $\epsilon\circ \iota\colon S^n(H^\ast(\Aa))\to \Aa^\ast$
to $\mathrm{Sym}^2(H^\bullet(A))$, and let
\[
\gamma\colon \mathrm{Sym}^2(H^\ast(\Aa^\ast))\to \Aa^{\ast-1}
\]
be the restriction of 
$d^{-}\circ \epsilon\circ \iota\colon S^n(H^\ast(\Aa^\ast))\to \Aa^{\ast-1}$ to $\mathrm{Sym}^2(H^\ast(\Aa))$.
The Bianchi-Massey tensor of $(\Aa^\ast,d,d^-)$ is the morphism
\[
\Xi_{BM}\colon \mathrm{Sym}^2(\mathrm{Sym}^2(H^\ast(\Aa)))\to H^\ast(\Aa)
\]
given by
\begin{equation}\label{eq:BM}
e\odot e' \mapsto [\pi_{\mathcal{H}}(\gamma(e)\cdot \alpha_2(e')+(-1)^{\deg e} \alpha_2(e)\cdot \gamma(e'))]
\end{equation}
\end{definition}
\begin{remark}
Notice that the definition of the Bianchi-Massey tensor does not make use of the orthogonality relations \eqref{eq:orthogonality}, but only of the fact $d^{-}$ is a homotopy between the identity of $\Aa^\ast$ and $\iota\circ \pi_{\mathcal{H}}$. So it can be defined for any Poincar\'e DGCA $\Aa^\ast$, since algebraic homotopies of this kind always exist. Here we assumed the homotopy is a Hodge one in Definition \ref{def:BM} only because that is the setting we are working in in the present article.
\end{remark}
\begin{remark}\label{rem:BM-CN}
As for Massey products, one can restrict $\Xi_{BM}$ to the subspace $K$ of $ \mathrm{Sym}^2(\mathrm{Sym}^2(H^\ast(\Aa)))$ spanned by those elements $e\odot e'$ such that both $e$ and $e'$ are in the kernel of the multiplication $\mathrm{Sym}^2(H^\ast(\Aa))\to H^\ast(\Aa)$, i.e., such that
$\epsilon(\iota((e))$ and $\epsilon(\iota((e'))$ are exact elements in $\Aa^\ast$. For these elements, $\gamma(e)\cdot \alpha_2(e'){+(-1)^{\deg e}} \alpha_2(e)\cdot \gamma(e')$ is a closed element in $\Aa^\ast$ and so, writing $\mathcal{F}$ for $\Xi_{BM}\vert_K$ gets the simplified formula $\mathcal{F}(e\odot e')= [\gamma(e)\cdot \alpha_2(e')\pm \alpha_2(e)\cdot \gamma(e')]$, which is the form the Bianchi-Massey tensor originally appears in \cite{CN}. Notice that{, thanks to Poincar\'e duality,} $\mathcal{F}$ retains all of the information on $\Xi_{BM}$. Indeed, { by Poincar\;e duality, $\Xi_{BM}$ is determined by its top degree component, i.e., by the elements $\Xi_{BM}(e\odot e')$ with $\deg(e)+\deg(e')=n+1$, where $n$ is the degree of $\Aa^\ast$.} If $a,b$ are elements in $H^\ast(\Aa)$, we have
\[
d^-(\iota(a)\iota(b)-\iota(ab))=d^-(\iota(a)\iota(b)),
\]
since $d^-$ vanishes on $\mathcal{H}^\ast$. Therefore, if we take $e$ to be the element $e=a\odot b-(ab)\odot 1$ in $\mathrm{Sym}^2(H^\ast(\Aa))$, then $e$ is in the kernel of the multiplication $\mathrm{Sym}^2(H^\ast(\Aa))\to H^\ast(\Aa)$ and moreover we have 
\[
\iota(e)=\iota(a)\odot\iota(b)-\iota(ab)\odot 1
\]
and so
\[
\gamma(e)=(d^{-}\circ \epsilon\circ \iota)(e)=d^-(\iota(a)\iota(b))=\gamma(a\odot b).
\]
{Similarly, let $a',b'\in H^\ast(\Aa)$ with $\deg(a')+\deg(b')+\deg(a)+\deg(b)=n+1$, and $e'=a'\odot b'-(a'b')\odot 1$. Then 
\[
\alpha_2(e')=\iota(a)\iota(b)-\iota(ab)=\alpha_2(a'\odot b')-\iota(a'b'),
\]
and so
\[
[\gamma(e)\alpha_2(e')]=[\gamma(a\odot b)\alpha_2(a'\odot b')]-[d^-(\iota(a)\iota(b))\, \iota(a'b')],
\]
where we used that we are in top degree to imply all the expressions above correspond to closed elements. By the orthogonality relations and Poincar\'e duality we have $[d^-(\iota(a)\iota(b))\, \iota(a'b')]=0$ and so $[\gamma(e)\alpha_2(e')]=[\gamma(a\odot b)\alpha_2(a'\odot b')]$. Switching the role of $a\odot b$ and $a'\odot b'$ one gets the identity
\[
\Xi_{BM}((a\odot b)\odot (a'\odot b'))=\Xi_{BM}(e\odot e')
\]
with $e, e'$ in the kernel of the multiplication $\mathrm{Sym}^2(H^\ast(\Aa))\to H^\ast(\Aa)$, so that $\Xi_{BM}(e\odot e')=\mathcal{F}(e\odot e')$.
}
\end{remark}
\begin{theorem}\label{thm:HarrBM}
Let $(\Aa^\ast,d,d^-)$ be a Poincar\'e DGCA endowed with a Hodge homotopy. The Bianchi-Massey tensor $\Xi_{BM}$ (or, more precisely, its extension to $\mathrm{Sym}^2(S^2(H^\ast(\Aa^\ast)))$ is a representative for the distinguished Harrison cohomology class $[\mu_3]$ of $(\Aa^\ast,d,d^-)$.
\end{theorem}
\begin{proof}
Since the ternary multiplication $\mu_3$ on $H^\ast(\Aa)$ is induced by the ternary multiplication $m_3$ on $\mathcal{H}^\ast$ via the linear isomorphism $\iota$, we see from \eqref{eq:m3} that it is explicitly given by
\[
\mu_3 (a, b, c) = [\pi_{\Hh} ( d^{-} (\iota(a) \cdot \iota(b)) \cdot \iota(c)  - (-1)^{\deg  a}  \iota(a) \cdot  d^-  (\iota(b) \cdot \iota(c)))].
\]
By \eqref{eq:BM} we see that for $x,y\in S^2(H^\ast(\Aa))$ we have
\[
\Xi_{BM}(x,y)=[\pi_{\mathcal{H}}(\epsilon((d^{-}\epsilon\iota(x))\odot \epsilon\iota(y))]+(-1)^{\deg x} [\pi_{\mathcal{H}}(\epsilon(\epsilon\iota(x)\odot (d^{-}\epsilon\iota(y)))]
\]
We have to check equation \eqref{eq:Harr-to-sym} for any three $a,b,c$ in $H^\ast(\Aa)$. Since $d^{-1}$ vanishes on $\mathcal{H}^\ast$, for any two $x,y\in H^\ast(\Aa)$ we have $\Xi_{BM}(x,y)=0$.
 Therefore, for any three $a,b,c\in A$,
\begin{align*}
d_{\mathrm{Harr}}&\Xi_{BM}(x,y,z)\\
&=-a\, \Xi_{BM}(b,c)+ \Xi_{BM}(a\odot b,c)- \Xi_{BM}(a,b\odot c)+ \Xi_{BM}(a,b)c\\
& =\Xi_{BM}(a\odot b,c)- \Xi_{BM}(a,b\odot c)\\
&=[\pi_\mathcal{H}(\epsilon(d^{-}(\epsilon\iota(a\odot b))\odot \epsilon\iota(c))]+(-1)^{\deg a\cdot b} [\pi_{\mathcal{H}}(\epsilon(\epsilon\iota(a\odot b)\odot d^{-}(\epsilon\iota(c)))]\\
&\qquad -[\pi_{\mathcal{H}}(\epsilon(d^{-}(\epsilon\iota(a))\odot \epsilon\iota(b\odot c))]-(-1)^{\deg a}[\pi_{\mathcal{H}}(\epsilon(\epsilon\iota(a)\odot d^{-}\iota(\epsilon(b\odot c)))]\\
&=[\pi_{\mathcal{H}}(\epsilon(d^{-}(\epsilon\iota(a\odot b))\odot \epsilon\iota(c))]- (-1)^{\deg  a} [\pi_{\mathcal{H}}(\epsilon(\epsilon\iota(a)\odot d^{-}\epsilon(\iota(b\odot c)))]\\
&=[\pi_{\mathcal{H}}((d^{-}(\iota(a)\cdot \iota(b))\cdot \iota(c))]- (-1)^{\deg  a} [\pi_{\mathcal{H}}(\iota(a)\cdot d^{-}(\iota(b)\cdot \iota(c)))]\\
&=\mu_3(a,b,c).
\end{align*}
\end{proof}

\begin{remark}\label{rem:HarrBM}
	
(1)  In \cite{CN}  Crowley and Nordstr\"om  also  discussed the relation  between  their  Bianchi-Massey  tensor on a  Poincar\'e DGA  $(\Aa^*, d)$  and the associated  $A_\infty$-structure on $H^* (\Aa^*)$ obtained  by a homotopy transfer  as  described  in \cite{Vallette14}. They noticed that  the multiplication  $\mu_3$ on  the $A_\infty$ algebra $H^* (\Aa^*)$  defines  the Bianchi-Massey tensor  uniquely and the  Bianchi-Massey    tensor  defines  $\mu_3$  up to a ``stabilization" \cite[Lemma 2.12]{CN}. Refer  to  \cite[\S 2.5]{CN}  for a  precise formulation of their result.

(2) Despite the Bianchi-Massey tensor $\Xi_{BM}$ (or its restriction $\mathcal{F}$, see Remark \ref{rem:BM-CN}) and the Harrison cohomology class $[\mu_3]$ contain the same amount of information, $\Xi_{BM}$ and $\mathcal{F}$ appear to be more suitable than $[\mu_3]$ for the investigation of questions related to rational and real homotopy types. One reason for this is that $\mathcal{F}$ enjoys a graded version of the symmetries of the Riemann curvature tensor (hence the name ``Bianchi", while the name``Massey" comes from its relation with the ternary multiplication $\mu_3$ and so with Massey products), so that in the analysis of $\mathcal{F}$ one can make full use of the vast amount of known algebraic results on the Riemann curvature tensor. For instance, in \cite{CN}, Crowley and Nordstr\"om make a brilliant use of the fact that  
the Ricci curvature of a manifold of dimension $n$ with $n\leq 3$ determines the full Riemann curvature tensor to prove that if $M$ is a closed $(r-1)$-connected manifold of dimension $n=(4r-1)$ with $b_{r}(M )\leq 3$ and such that there
is a $\phi \in H^{2r-1}_{\mathrm{dR}}(M )$ such that multiplication by $\phi$ induces a linear isomorphism $H^r_{\mathrm{dR}}(M ) \to H^{3r-1}_{\mathrm{dR}}(M)$, then $M$ is
intrinsically formal. Another truly remarkable result in \cite{CN} is the fact that, by using  the Bianchi-Massey tensor, Crowley and Nordstr\"om are able to improve the dimension bound $n\leq 4r-1$ in Theorem \ref{thm:complete-invariant} (that was also the bound in the first arXiv version of \cite{CN}) to the full range $n\leq 5r-3$ of the vanishing of $\mu_4$. Theorem \ref{thm:complete-invariant} with the optimal bound $n\leq 5r-3$ appears in \cite{FKLS2021} as Corollary 4.11. Yet, it is fair to say that the proof 
of Corollary 4.11 in \cite{FKLS2021} is incomplete and rather corresponds to the proof of Lemma \ref{lemma:first-obs-formality} in the present article than to a complete proof of Theorem \ref{thm:complete-invariant} with the optimal bound $n\leq 5r-3$. It would be interesting to understand whether one can give a proof of this result entirely in terms of the Harrison cohomology class $[\mu_3]$ alone, as the proof we present here for $n\leq 4r-1$, without relying on its equivalence with $\Xi_{BM}$ or $\mathcal{F}$.

\end{remark}

\section{Conclusions and final remarks}\label{sec:conclusions}
\begin{enumerate}
\item In this  paper,  using Hodge homotopy transfers, we  gave a very quick proof of Zhou's result   that   the  real  homotopy type  of a $(r-1)$-connected  compact  manifold  $M$ of dimension  $n \le \ell(r-1)+2$, with $\ell\geq 4$, is encoded  in a  minimal unital  $C_\infty$-algebra $(H^* (M,\R), \mu_2, \mu_3, \ldots)$ whose multiplication $\mu_2$ is  the  usual multiplication  of the cohomology  alegebra $H^*  (M, \R)$ and $\mu_k$ vanishes  for $k \ge \ell-1$ (Theorem \ref{thm:A2-algebra}).  Also, we extended Zhou's result, that is originally given in terms of $A_\infty$-algebras, to $C_\infty$-algebras.

 \item  We gave a    proof of Zhou's  conjecture, stating  that  the vanishing   in Theorem \ref{thm:A2-algebra} also holds  if  we  increase  the dimension  of $M$  up to  2, under the  condition that  $b_r =1$  (Theorem \ref{thm:zhou-s-conjecture}).  If $\ell =4$,  this  is  a  generalization of  Cavalcanti's  formality  result  in \cite{Cavalcanti2006}.  

\item   Let $M$ be an arbitrary   simply connected   compact manifold.  We  showed  that the  Harrison  cohomology class    $[\mu_3] \in  {\rm HHarr}^{3,-1} (H^* (M,\R), H^* (M, \R)) $  is   an invariant  of  the homotopy type  of  $M$  and   it is  the first obstruction    to the formality of  $M$.  If $M$ is  a $(r-1)$ connected  compact manifold  of dimension  $n \le 4r-1$ then $[\mu_3]$ is  moreover a complete  invariant  of    the real  homotopy type of  $M$. These    results  follow  from   more    general   Lemma \ref{lemma:first-obs-formality}  and  Theorem \ref{thm:complete-invariant}.

\item  We  relate  our    above results  with  Crowley-Nordstr\"om's   results   on Bianchi-Massey  tensor \cite{CN}  and  Nagy-Nordstr\"om's results  on  Bianchi-Massey and pentagon  Massey tensor  \cite{NN2021} by  showing  that 
the  Harrison  cohomology class  $[\mu_3]$  and  the   Bianchi-Massey  tensor  define  each other  uniquely  (Theorem \ref{thm:HarrBM}).  Using  this result,   our   Theorem \ref{thm:complete-invariant} can be obtained  from Crowley-Nordstr\"om results \cite[Theorems 1, 2]{CN},  but Theorem \ref{thm:A2-algebra} and Lemma \ref{lemma:first-obs-formality}  don't have   direct counterparts in Crowley-Nordstr\"om   paper \cite{CN} and  in  Nagy-Nordstr\"ome paper \cite{NN2021}, respectively.

\item   Theorem \ref{thm:HarrBM}  and   Crowley-Nordstr\"om result \cite[Theorem 1.2]{CN}   suggest that we should  be able to  improve  the  proof of Theorem \ref{thm:complete-invariant} to increase  the dimension $4r-1$ to  $5r-3$.

\item  Taking  into account  Kadishvili's results on  homotopy type  of a $C_\infty$-algebra \cite[Proposition 17]{Kadeishvili2023}, \cite[\S 4.1]{Kadeishvili2007}, Theorem \ref{thm:A2-algebra}  suggests  that  the real homotopy type  of a $(r-1)$ simply connected   compact manifold of dimension  $n \le 6r-4$ is encoded  in  $[\mu_3]$ and a  ``canonical" cochain in ${\rm CHarr}^{4, -2}(H^* (M, \R), H^* (M, \R))$. This conjecture  agrees  with and extends Nagy-Nordstr\"om's  conjecture  \cite[Conjecture 1.9]{NN2021} that in  the case  the Bianchi-Massey  tensor vanishes,  the  rational homotopy type  of a  $(r-1)$-connected  compact  manifold $M$ of dimension less or equal to $5r-2$ is defined by the
pentagon  Massey  tensor.
\end{enumerate} 

\subsection*{Acknowledgement}
We would like to thank  Johannes Nordstr\"om  for alerting us  about his  eprint with Nagy \cite{NN2021}, for his  inspiring  lecture \cite{Nordstroem2023}, which motivated us  to start  this project{, and for useful comments on a preliminary version of this article}. We thank Fernando Muro for having pointed our attention to relevant works concerning the Harrison cohomology class $[\mu_3]$ and its role as an obstruction to formality.  {We thank the referee for careful reading and precious suggestions.} DF  is  thankful to  the  Institute  of Mathematics  of the Czech Academy of Sciences  and   the  Charles  University  for their  hospitality  and   excellent working  conditions  during  his  visiting  Prague in   August 2023, where a part  of this   project  has been  carried out.  {This paper is part of the activities of the MIUR Excellence Department
Projects CUP:B83C23001390001.
DF has been partially supported by the PRIN 2017 {\em Moduli and Lie
Theory}, Ref.\ 2017YRA3LK. He is a member of the {\em Grup\-po Nazionale per le
Strutture Algebriche, Geometriche e le loro Applicazioni} (GNSAGA-INdAM). Research  of HVL was  supported by the Institute of Mathematics, Czech Academy of Sciences (RVO 67985840), and  GA\v CR-project  GA22-00091S.}


\bibliographystyle{unsrt}

\begin{thebibliography}{99999}
\bibitem{Barr1968}{\sc M. Barr}, {\it Harrison homology, Hochschild homology and triples}, J. Algebra, 8, 314-323, (1968).
\bibitem{BKS2003}{\sc D. Benson, H. Krause, and S. Schwede}, {\it Realizability of modules over Tate cohomology},  Trans. Amer. Math. Soc. 356 (2003),  Nr 9, 3621-3668.

\bibitem{Cavalcanti2006}{\sc Gil R. Cavalcanti}, {\it Formality of $k$-connected spaces in $4k + 3$ and $4k + 4$ dimensions}, Math.
Proc. Cambridge Philos. Soc. 141 (2006), 101-112.
\bibitem{CG2008}{\sc  X. Z. Cheng, E. Getzler} 
{\it Transferring homotopy commutative algebraic structures}, J. Pure Appl. Algebra 212 (2008) 2535-2542.

\bibitem{CN} {\sc D. Crowley, J. Nordstr\"om}, {\it The rational homotopy type of $(n-1)$-connected manifolds of dimension up to $5n-3$}, J. Topol.  13(2020), 539-575, arXiv:1505.04184v3.
\bibitem{DGMS1975}{\sc P. Deligne, P. Griffiths, J. Morgan and D. Sullivan}, {\it Real Homotopy Theory of K\"ahler manifolds},  Invent. Math. 29(1975), 245-274.

\bibitem{GCTV2012}{\sc I. G\'alvez-Carrillo, A. Tonks, and B. Vallette}, {\it Homotopy Batalin-Vilkovisky Algebras},
J. Noncommut. Geom. 6 (2012), 539-602.
\bibitem{GH1978}{\sc P. Griffiths and J. Harris}, ``Principles of algebraic geometry'', John Willey and Sohns,  1978. 

\bibitem{FOT2008}{\sc Y. Felix, J. Oprea and  D. Tanr\'e}, ``Algebraic Models in Geometry'', Oxford University Press, 2008.

\bibitem{FM2005}{\sc  N. Fernandez V. Munos},  {\it Formality  of Donaldson  submanifolds},  Math. Z.  250 (2005), 149-175.

\bibitem{FKLS2021}{\sc D. Fiorenza, K. Kawai, H. V. L\^e and L. Schwachh\"ofer},  {\it Almost formality  of manifolds of low dimension},
Ann. Sc. Norm. Super. Pisa Cl. Sci. (5), vol. XXII (2021), 79-107.

\bibitem{Fleury1973} {\sc P. Fleury}, {\it Symmetric  Homology over rings containing  the rationals}, Proc. Amer. Math. Soc. vol. 40, (1973), 69-74.

\bibitem{graves}{\sc D. Graves}, {\it A Comparison Map for Symmetric Homology and Gamma Homology}, {\tt arXiv:2005.07963v2}.


\bibitem{GM2013} {\sc P. Griffiths, J. Morgan},  ``Rational Homotopy Theory and Differential  Forms'',  Brikh\"auser,  2nd  Edition, 2013.



\bibitem{JKM2022}{\sc G. Jaso  F. Muro} {\it The derived Auslander-Iyama  Correspondence} (with Appendix  by B. Keller) arXiv:2208.14413.
\bibitem{Kadeishvili1980}{\sc T. V. Kadeishvili},  {\it On the homology theory of fibre spaces}, Russian Math. Surveys, 35(3):231-238, 1980.
\bibitem{Kadeishvili1988} {\sc T. V. Kadeishvili}, {\it Structure of $A(\infty)$-algebra and Hochschild and Harrison cohomology}, Proc. A. Razmadze Math. Inst., 91 (1988), 20-27.
\bibitem{Kadeishvili1993} {\sc T. V. Kadeishvili}, {\it $A_\infty$-algebra Structure in Cohomology and Rational Homotopy Type}, (in Russian), Proc. Tbil. Mat. Inst., 107, (1993), 1-94.
\bibitem{Kadeishvili2007}{\sc T. V. Kadeishvili},  {\it Twisting Elements in Homotopy $G$-algebra},
arXiv:0709.3130.
\bibitem{Kadeishvili2009}{\sc T. V. Kadeishvili},  {\it Cohomology $C_\infty$-algebra and rational homotopy type}, In: ``Algebraic Topology, Old and New", Banach Center Publ., Vol. 85, Polish Acad. Sci. Inst. Math.,
Warsaw, 2009, 225-240.
\bibitem{Kadeishvili2023}{\sc T. V. Kadeishvili}, {\it $A_\infty$-algebra Structure in Cohomology and its Applications}, arXiv:2307.10300.
\bibitem{KS2001}{\sc M. Kontsevich, Y. Soibelman}, 
{\it Homological mirror symmetry and torus fibrations}, Symplectic geometry and mirror symmetry (Seoul, 2000), 203-263, World Sci. Publ., River Edge, NJ, 2001.

\bibitem{KT1991}{\sc M. Kreck and G. Triantafillou}, {\it On the classification of manifolds up to finite ambiguity}, Canad. J. Math. 43(1991), no. 2, 356-370.

\bibitem{LS04}{\sc P. Lambrechts and D. Stanley}, {\it  The rational homotopy type of configuration spaces of two points}, Ann. Inst. Fourier (Grenoble) 54 (2004), no. 4, 1029--1052.

\bibitem{loday}{\sc J.-L. Loday}, ``Cyclic homology'', Springer, 1992, (Second ed.: 1998).


\bibitem{loday-vallette}{\sc J.-L. Loday and B. Vallette}, ``Algebraic operads'', Grundlehren Math. Wiss. 346, Springer, Heidelberg, 2012.



\bibitem{Merkulov1999}{\sc S. A. Merkulov}, {\it Strong homotopy algebras of a K\"ahler manifold}, 
Int. Math. Res. Not. IMRN    (1999), Nr 3,  153-164.


\bibitem{Miller1979}{\sc T. J. Miller},  {\it On the formality of $k - 1$ connected compact manifolds of dimension less than or equal to $4k - 2$},
Illinois J. Math., 23 (1979), 253-258.


  
\bibitem{NN2021}{\sc C. Nagy and  J. Nordstr\"om},  {\it Rational Homotopy theory and  simply-connected 8-manifolds}, arXiv:2105.13660.

\bibitem{Nordstroem2023}{\sc  J. Nordstr\"om}, \href{https://people.bath.ac.uk/jlpn20/masseyt.pdf}{{\it Massey tensors and the rational homotopy type of 7- and 8-manifolds}},  Lecture at PHK-Cohomology seminar, May 3, 2023.

\bibitem{Smale1961}{\sc S. Smale}, {\it Generalized Poincar\`e's conjecture in dimensions greater than four}, Ann. of Math. (2)74(1961), 391-406.

\bibitem{Smale1962}{\sc S. Smale}, {\it On the structure of $5$-manifolds},  Ann. of Math. (2) 75 (1962), 38-46. 


\bibitem{Sullivan1977}{\sc D. Sullivan}, 
{\it Infinitesimal computations in topology}, Publ. Math. Inst. Hautes \'Etudes Sci., 47 (1977), 269-331.

\bibitem{Sullivan1975}{\sc D.  Sullivan}, {\it Differential forms and the topology of manifolds}, Manifolds-Tokyo (1973) (Proc. of the Intern. Conf. on Manifolds 
and related topics in Topology, Tokyo 1973) (ed. A. Hattori), U. of
Tokyo Press, 1975, 37-49.



\bibitem{Tanre1985}{\sc D. Tanre,} {\it Cohomology de Harrison et type d'homotopy rationelle},  J. Pure Appl. Algebra, 38(1985), 353-366.

\bibitem{VdL2003} {\sc P. Van der Laan},  {\it Coloured Koszul duality and strongly homotopy operads},
arXiv:math.QA/0312147, 2003.


\bibitem{Vallette14}{\sc B. Vallette}, {\it Algebra + homotopy = operad}, in: ``Symplectic, Poisson and Noncommutative Geometry", Math. Sci. Res. Inst. Publ., Vol. 62, Cambridge Univ. Pres, New York, 2014, 229-290.

\bibitem{Zhou2019} {\sc J. Zhou}, {\it $A_\infty$-Minimal Model on Differential Graded Algebras}, arXiv:1904.10143.

\end{thebibliography}

\end{document}